\documentclass[final]{amsart}
\usepackage[automark]{scrpage2}
\usepackage[ngerman, english]{babel}
\usepackage[backref=page,final]{hyperref}
\usepackage{lmodern}
\usepackage[leqno]{amsmath}
\usepackage{amssymb,amscd}
\usepackage{amsthm}
\usepackage{amsopn}
\usepackage[all]{xypic}
\usepackage{mathtools}
\usepackage{bbm}
\usepackage{listings}
\usepackage{verbatim}
\usepackage{braket}
\ifnum\pdfoutput=0
\usepackage[final]{graphicx}
\else
\usepackage[pdftex,final]{graphicx}
\usepackage{epstopdf}
\fi
\renewcommand{\emph}[1]{{\bf #1}}

\newcommand{\F}{\mathbb{F}}

\newcommand{\C}{\mathbb{C}}
\newcommand{\N}{\mathbb{N}}
\newcommand{\Z}{\mathbb{Z}}
\newcommand{\Q}{\mathbb{Q}}

\newcommand{\Proj}{\mathbb{P}}

\newcommand{\Orbit}{\mathcal{O}}

\newtheorem{definition}{Definition}[section]
\newtheorem{lemma}[definition]{Lemma}
\newtheorem{corollary}[definition]{Corollary}
\newtheorem{theorem}[definition]{Theorem}
\newtheorem{proposition}[definition]{Proposition}
\theoremstyle{remark}
\newtheorem{example}[definition]{Example}
\newtheorem{remark}[definition]{Remark}

\DeclareMathOperator{\rank}{rank}

\DeclareMathOperator{\GL}{GL}
\DeclareMathOperator{\Tr}{Tr}

\DeclareMathOperator{\End}{End}

\DeclareMathOperator{\Hom}{Hom}
\DeclareMathOperator{\Ext}{Ext}

\DeclareMathOperator{\pd}{pd}
\DeclareMathOperator{\gldim}{gldim}

\DeclareMathOperator{\ext}{ext}
\DeclareMathOperator{\Ker}{Ker}
\DeclareMathOperator{\Bild}{Im}
\DeclareMathOperator{\Mod}{mod}
\DeclareMathOperator{\MMod}{Mod}

\DeclareMathOperator{\id}{id}

\DeclareMathOperator{\Mat}{Mat}

\DeclareMathOperator{\codim}{codim}

\DeclareMathOperator{\Spec}{Spec}

\DeclareMathOperator{\OGr}{Gr}
\DeclareMathOperator{\OFl}{Fl}
\DeclareMathOperator{\ORep}{Rep}
\DeclareMathOperator{\Orep}{rep}

\DeclareMathOperator{\ORepFl}{RepFl}
\DeclareMathOperator{\ORepHom}{RepHom}
\newcommand{\Sch}[1]{\mathrm{#1}}

\newcommand{\flvec}[1]{\seqv{\dvec{#1}}}
\newcommand{\tensor}{\otimes}
\newcommand{\qbinom}[2]{\begin{bmatrix}#1\\#2\end{bmatrix}}

\newcommand{\repK}[2]{\ensuremath{\Orep(#1, #2)}}

\newcommand{\Rep}[2][]{\ensuremath{\ORep_{#1}({#2})}}

\newcommand{\RepK}[3][]{\ensuremath{\ORep_{#1}({#2}, {#3})}}
\newcommand{\Gr}[3][]{\ensuremath{\OGr_{#1}\binom{#3}{#2}}}

\newcommand{\RepFl}[2][]{\ensuremath{\ORepFl_{#1}\left({#2}\right)}}
\newcommand{\RepHom}[2][]{\ensuremath{\ORepHom_{#1}\left({#2}\right)}}
\newcommand{\Fl}[3][]{\ensuremath{\OFl_{#1}\binom{#3}{#2}}}

\newcommand{\euf}[1]{\left\langle {#1} \right\rangle}
\newcommand{\lrang}[1]{\left\langle {#1} \right\rangle}
\newcommand{\bform}[2]{\lrang{\,{#1}\, {,}\, {#2}\,}}
\newcommand{\sbform}[2]{\left({#1}\, {,}\, {#2}\right)}
\newcommand{\dvec}[1]{{\underline{#1}}}
\newcommand{\ledeg}{\le_{\mathrm{deg}}}

\newcommand{\degr}[1]{\lvert {#1} \rvert}
\newcommand{\dimv}[1]{\underline{\dim}\left(#1\right)}
\newcommand{\dimve}[1][]{\underline{\dim}\;{#1}}
\newcommand{\seqv}[1]{{\boldsymbol{#1}}}
\newcommand{\rever}[1]{\ensuremath\stackrel{\leftarrow}{#1}}

\def\clap#1{\hbox to 0pt{\hss#1\hss}}
\def\mathclap{\mathpalette\mathclapinternal}
\def\mathclapinternal#1#2{%
	\clap{$\mathsurround=0pt#1{#2}$}%
}

\usepackage{rotating}
\usepackage{setspace}
\usepackage{tikz}
\hypersetup{
    colorlinks,%
    citecolor=black,%
    filecolor=black,%
    linkcolor=black,%
    urlcolor=black
}
\usetikzlibrary{matrix,arrows,decorations.pathmorphing,calc}
\author{Stefan Wolf}
\date{}
\address{%
Institut f\"ur Mathematik, Universit\"at Paderborn, Warburger Str. 100, 33098 Paderborn
}
\email{swolf@math.upb.de}
\title{A geometric version of BGP reflection functors}
\begin{document}
\newbox\ASYbox
\newdimen\ASYdimen
\def\ASYbase#1#2{\leavevmode\setbox\ASYbox=\hbox{#1}\ASYdimen=\ht\ASYbox%
\setbox\ASYbox=\hbox{#2}\lower\ASYdimen\box\ASYbox}
\def\ASYalign(#1,#2)(#3,#4)#5#6{\leavevmode%
\setbox\ASYbox=\hbox{#6}%
\setbox\ASYbox\hbox{\ASYdimen=\ht\ASYbox%
\advance\ASYdimen by\dp\ASYbox\kern#3\wd\ASYbox\raise#4\ASYdimen\box\ASYbox}%
\put(#1,#2){%
\wd\ASYbox 0pt\dp\ASYbox 0pt\ht\ASYbox 0pt\box\ASYbox%
}}
\def\ASYraw#1{#1}
\bibliographystyle{amsalphaurl}
\SelectTips{eu}{10}
\begin{abstract}
Quiver Grassmannians and quiver flags are natural generalisations of usual Grassmannians
and flags. They arise in the study of quiver representations and Hall
algebras. In general, they are projective varieties which are neither smooth nor irreducible.

We use a scheme theoretic approach to calculate their tangent space
and to obtain a dimension estimate similar to the one of Reineke in \cite{Reineke_monoid}.
Using this we can show that if there is a generic representation, then these varieties are smooth
and irreducible. If we additionally
have a counting polynomial we deduce that their Euler characteristic is positive
and that the counting polynomial evaluated at zero yields one.

After having done so, we introduce a geometric version of BGP reflection functors
which allows us to deduce an even stronger result about the constant coefficient of
the counting polynomial. We use this to obtain an isomorphism between the
Hall algebra at $q=0$ and Reineke's generic extension monoid in the Dynkin case.
\end{abstract}
\maketitle
One of the first non-trivial varieties one studies when learning algebraic geometry is
the Grassmannian. It consists of $r$-dimensional subspaces of a fixed $d$-dimensional
vector space $V$. The (vector space) Grassmannian is smooth and irreducible, has a nice
functor of points and the tangent space at an $r$-dimensional subspace $U$ of $V$
is given by linear maps from $U$ to $V/U$. A first generalisation of this variety is
the flag variety, which consists of a filtration of $V$ by subvector spaces of fixed
dimension vectors.

Although the vector space Grassmannian already leads to interesting geometrical problems,
it is still very easy. A generalisation is given by taking quiver Grassmannians. A quiver $Q$
is an oriented graph and a $k$-representation of $Q$ is given by assigning a $k$-vector
space to each vertex of $Q$ and a $k$-linear map between these vector spaces to each
arrow of $Q$. The dimension vector of a finite dimensional $k$-representation M
is the tuple of dimensions of the vector spaces at the vertices.
For a $k$-representation $M$ the quiver Grassmannian
consists of subrepresentations of $M$ of dimension vector $\dvec{d}$.
In general, this scheme is much more complex, or interesting, as
the vector space Grassmannian. It is generally neither smooth nor reduced nor irreducible.
Studying the quiver Grassmannian was first done by Schofield \cite{Schofield_genrep}
when he introduced a calculus of Schur roots. Recently it received attention for
example by Caldero and Chapoton \cite{CalderoChapoton_clusterhall} when they
showed that its Euler characteristic gives coefficients in the cluster algebra.
Caldero and Reineke \cite{CalderoReineke_QuiverGrassmannian} showed that the Euler characteristic is positive provided
that the quiver is acyclic and that the representation is rigid by using
Lusztig's pervese sheaf approach \cite{Lusztig_quantum} to quantum groups.
There is also the natural generalisation to quiver flags, but up to now there has been
no work investigating geometric properties of those.

The first part of this paper deals with a geometric analysis of quiver flags and quiver Grassmannians.
We introduce their functors of points and calculate their tangent spaces. In order to do this
for the quiver flag we use a generalisation of tensor algebras. Then we use this
information and Chevalley's theorem to obtain a dimension estimate for the quiver flag similar
to the one of Reineke \cite{Reineke_monoid} for quiver Grassmannians, but here it holds not only in characteristic $0$.
Afterwards, we use this to show smoothness and irreducibility of the quiver flag if there
is a generic representation. Finally, under this assumption and additionally having a counting polynomial,
we obtain the positivity of the Euler characteristic and
that the constant term of the counting polynomial is $1$.

There is Ringel's Hall algebra approach to quantum groups \cite{Ringel_hallalgsandquantumgroups}.
He generalised a construction of Steinitz \cite{Steinitz_abelsch} and
Hall \cite{Hall_part}
to obtain an associative algebra structure
on the $\Q$-vector space $\mathcal{H}_k(Q)$ with
basis the isomorphism classes of representations of $Q$ over a finite field $k$, the Ringel-Hall algebra.
The structure constants are given by counting numbers of subrepresentations having a fixed isomorphism class such that the
quotient also has a fixed isomorphism class.

The whole Ringel-Hall algebra is generally too complicated, and therefore one introduces the $\Q$-subalgebra
$\mathcal{C}_k(Q)$ generated by the isomorphism classes of simple representations without self-extensions.
The coefficients of a representation in the product of a sequence of simples is given by the number of
$k$-valued points of a certain quiver flag variety, therefore counting rational points of these varieties is
crucial there.
There is a generic version $\mathcal{C}_q(Q)$, a $\Q[q]$-algebra, such that
specialising $q$ to $|k|$ recovers $\mathcal{C}_k(Q)$.
Ringel's paper \cite{Ringel_hallalgsandquantumgroups}, for the Dynkin case, and later
Green's paper \cite{Green_hallalgs} combined with a result of Sevenhant and Van Den Bergh \cite{SevenhantVanDenBergh_kachall}, for the general case,
yield that, after twisting the multiplication with
the Euler form of $Q$, the generic composition algebra $\mathcal{C}_{q}(Q) \tensor \C(q)$
is isomorphic
to the
positive part of the quantised enveloping algebra of the Kac-Moody
Lie algebra corresponding to $Q$.

Reineke \cite{Reineke_monoid} showed that one has a monoid structure on
the set of irreducible closed subvarieties of the representation
variety. The multiplication is
given by taking all possible extensions. This monoid is called
the generic extension monoid $\mathcal{M}(Q)$. Similarly to the Hall algebra,
it is in general too complicated, so one restricts itself to the
submonoid generated by the orbits of simple representations without self-extensions,
the composition monoid $\mathcal{CM}(Q)$.
The elements of the composition monoid are the varieties consisting of representations
having a composition series with prescribed composition factors in prescribed order
and are therefore closely related to quiver flags.

For the generic composition algebra viewed as a $\Q(v)$-algebra with $v^2 =q$
the (twisted) quantum Serre relations are defining as shown in \cite{Ringel_greenstheom}. Reineke
showed that the quantum Serre relations specialised to
$q=0$ hold in the composition monoid. They are in general not defining any more if we
specialise $q$ to $0$ in the composition algebra. But nonetheless one can conjecture,
as Reineke did in \cite{Reineke_genericexts} and \cite{Reineke_monoid}, that there is
a homomorphism of $\Q$-algebras
 \[\Phi \colon \mathcal{C}_0(Q) \rightarrow \Q\mathcal{CM}(Q)\]
sending simples to simples and therefore being automatically surjective.

The first step in this direction was done by Hubery \cite{Hubery_kronecker} showing that for
the Kronecker quiver
\[K =
\begin{array}{@{}c@{}}
\makeatletter%
\let\ASYencoding\f@encoding%
\let\ASYfamily\f@family%
\let\ASYseries\f@series%
\let\ASYshape\f@shape%
\makeatother%
\setlength{\unitlength}{1pt}
\includegraphics{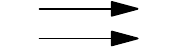}%
\definecolor{ASYcolor}{gray}{0.000000}\color{ASYcolor}
\fontsize{12.000000}{14.400000}\selectfont
\usefont{\ASYencoding}{\ASYfamily}{\ASYseries}{\ASYshape}%
\ASYalign(-47.027048,6.836534)(-0.500000,-0.500000){1.000000 0.000000 0.000000 1.000000}{$\bullet$}
\ASYalign(-4.187913,6.836534)(-0.500000,-0.500000){1.000000 0.000000 0.000000 1.000000}{$\bullet$}
\end{array}
\]
$\Phi$ is a homomorphism of $\Q$-algebras with non-trivial kernel. He did this by
calculating defining relations for $\mathcal{C}_0(K)$ and $\mathcal{CM}(K)$. He also was
able to give generators for the kernel of $\Phi$.

In order to get further results in this direction, we need to calculate many of the structure
coefficients in the composition algebra. These are given by counting
$k$-valued points of quiver flags of a representation
$M$ over finite fields $k$. For this, we develop
a framework for applying reflection functors
to quiver flags. We give a stratification of the flag into locally closed subschemes such
that on each strata we can nicely apply reflection functors. We will see that this
operations does not change the number of $k$-valued points modulo the number of elements of the field $k$.
Using this, we can immediately show that for the Dynkin
case the coefficients at $q=0$ are one or zero and by this we obtain immediately that
$\mathcal{C}_0(Q)$ and $\Q\mathcal{CM}(Q)$ are isomorphic.
\section{Geometric setup}
We want to define some functors from commutative rings to sets
which are schemes, ending up with the quiver Grassmannian. These should turn out to be
useful to define morphisms between schemes naturally coming up in representation theory of quivers.
All of this section should be standard. We say that a scheme $X$ over some field $k$ is a variety
if it is separated, noetherian and of finite type over $k$. The main reference for the
geometric methods is \cite{DG}.

For a commutative ring $R$ and an $R$-module $P$ we say that $P$ has rank $r$ if for every prime
$\mathfrak{p} \in \Spec R$ we have that the localisation $P_\mathfrak{p}$ is free of rank $r$ as an $R_\mathfrak{p}$-module.
As a consequence $P$ is projective.
Let $d, e, r \in \N$. We define various schemes via their functors of points from commutative rings to sets. Let $R$ be a commutative
ring. We define
\begin{align*}
    \Sch{Hom}(d,e) (R) \; &:= \Hom_R (R^d, R^e),\\
    \Sch{Hom}(d,e)_r(R)\; & := \Set{ f\in \Sch{Hom}(d,e) (R) |
    \begin{array}{l}
        \Bild f \text{ is a direct summand}\\ \text{of }R^e \text{ of } \rank r
    \end{array}},\\
    \Sch{Inj}(d,d+e)(R)\; & := \Set{ f\in \Sch{Hom}(d,d+e) (R) | f \text{ is a split injection}}\\
    & \; = \Sch{Hom}(d,d+e)_d(R)\\
    \intertext{and}
    \Sch{Surj}(d+e,e)(R) \; & := \Set{ g\in \Sch{Hom}(d+e,e) (R)| g \text{ is surjective} }\\
    & \; = \Sch{Hom}(d+e,e)_e(R).
\end{align*}
We define the Grassmannian scheme via its functor of points
    \[\Gr{d}{d+e} (R) := \Set{ P \subset R^{d+e} | P \text{ is a direct summand of }
    \rank d}.\]
    If $\seqv{d} = (d^0,\dots,d^\nu)$ is a sequence of integers, we define the flag
    scheme $\OFl(\seqv{d}) \subset \prod \Gr{d^i}{d^\nu}$
    as
 \[\OFl(\seqv{d}) (R):= \Set{ (U^0, \dots, U^\nu) \in \prod_{i=0}^\nu \Gr{d^i}{d^\nu}(R) |
 \begin{array}{c}
     U^i \subset U^{i+1} \text{ for}\\ \text{all } 0 \le i < \nu
 \end{array}}.\]
    Let $n \in \N$. We define the group scheme $\GL_n$ via its functor of points
    \[ \GL_n (R) := \GL(R^n). \]
\begin{lemma}
    Let $d, e, r \in \N$ and $\seqv{d} = (d^0,\dots,d^\nu)$ a sequence of integers. Then
    \begin{enumerate}
        \item $\Sch{Hom}(d,e)$ is an affine space.
        \item $\Sch{Hom}(d,e)_r$ is a locally closed subscheme of $\Sch{Hom}(d,d+e)$.
        \item $\Sch{Inj}(d,d+e)$ is an open subscheme of $\Sch{Hom}(d,d+e)$.
        \item $\Sch{Surj}(d+e,e)$ is an open subscheme of $\Sch{Hom}(d+e,e)$.
        \item $\Gr{d}{d+e}$ and $\OFl(\seqv{d}) \subset \prod \Gr{d^i}{d^\nu}$
            are projective schemes, smooth over $\Z$.
        \item $\GL_n$ is a smooth group scheme over $\Z$.
        \item The projection $\Sch{Inj}(d,d+e) \rightarrow \Gr{d}{d+e}$ sending each
    $f$ to $\Bild(f)$ is a principal $\GL_d$-bundle, locally trivial in the
    Zariski topology.
    \end{enumerate}
\end{lemma}
\begin{proof}
    For (1)--(4) and (7) see, for example, \cite{bill_homsgenrep}. For (5) and (6)
        see \cite{DG}.
\end{proof}

Now we generalise these schemes to quivers.
For standard notations and results about quivers we refer the reader to
\cite{Ringel_integral}.
A quiver $Q=(Q_0, Q_1, s, t)$ is a directed graph with a set of vertices $Q_0$, a set of
arrows $Q_1$ and maps
$s,t \colon Q_1 \rightarrow Q_0$, sending an arrow to its starting respectively terminating vertex.
In particular, we write $\alpha \colon s(\alpha) \rightarrow t(\alpha)$ for an $\alpha \in Q_1$. A quiver $Q$ is finite
if $Q_0$ and $Q_1$ are finite sets. In the following, all quivers will be finite.
For a quiver $Q=(Q_0, Q_1, s ,t)$ we define $Q^{op} := (Q_0, Q_1, t, s)$ as the quiver with all arrows reversed.

A path of length $r\ge0$ in $Q$ is a sequence of arrows $\xi = \alpha_1 \alpha_2 \cdots \alpha_r$ such that
$t(\alpha_i) = s(\alpha_{i+1})$ for all $1 \le i < r$. We write $t(\xi) := t(\alpha_r)$ and $s(\xi):= s(\alpha_1)$.

For each $i \in Q_0$ there is the
trivial path $\epsilon_i$ of length zero starting and terminating in the vertex $i$. We identify the elements of $Q_0$ with
paths of length $0$. We denote by $Q(i,j)$ the
set of paths starting at the vertex $i$ and terminating at the vertex $j$.

If $R$ is a ring, then the path algebra $RQ$ has as basis the set of
paths and the multiplication $\xi \cdot \zeta$ is given by
the concatenation of paths if $t(\xi)=s(\zeta)$ or zero otherwise. In particular, the $\epsilon_i$ are pairwise orthogonal
idempotents of $RQ$, that is $\epsilon_i \epsilon_j = \delta_{ij} \epsilon_i$.

Let $Q_r$ denote the set of paths of length $r$. This extends the notation of vertices $Q_0$ and arrows $Q_1$. Then
\[ RQ = \bigoplus_{r\ge 0} RQ_r,\]
where $RQ_r$ is the free $R$-module with basis the elements of $Q_r$. By construction,
\[RQ_r RQ_s = RQ_{r+s},\]
thus
$RQ$ is an $\N$-graded $R$-algebra.

We define a partial order on
$\Z Q_0$ by $\dvec{d} = \sum_{i} d_i \epsilon_i \ge 0$ if and only if $d_i \ge 0$ for all $i\in Q_0$.
An element $\dvec{d} \in \Z Q_0$ is called a dimension vector.
We endow $\Z Q_0$ with a bilinear form $\bform{\cdot}{\cdot}_Q$ defined by
\[
\bform{\dvec{d}}{\dvec{e}}_Q := \sum_{i \in Q_0} d_i e_i - \sum_{\mathclap{\alpha \colon i \rightarrow j \in Q_1}} d_i e_j.
\]
This form is generally called the Euler form or the Ringel form.
We also define its symmetrisation
\[
\sbform{\dvec{d}}{\dvec{e}}_Q :=\bform{\dvec{d}}{\dvec{e}}_Q + \bform{\dvec{e}}{\dvec{d}}_Q.
\]

Let $Q$ be a quiver and $k$ be a field. A $k$-representation
$M$ of $Q$ is given by finite dimensional $k$-vector spaces
$M_i$ for each $i\in Q_0$ and $k$-linear maps $M_\alpha \colon M_i \rightarrow M_j$ for each $\alpha \in Q_1$.
If $M$ and $N$ are $k$-representations of $Q$, then a morphism $f \colon M \rightarrow N$ is given by
$k$-linear maps
$f_i \colon M_i \rightarrow N_i$ for each $i \in Q_0$ such that the following diagram commutes
\[
\xymatrix{
M_i \ar[r]^{M_\alpha} \ar[d]_{f_i} & M_j \ar[d]^{f_j}\\
N_i \ar[r]_{N_\alpha} & N_j
}
\]
for all $\alpha \colon i \rightarrow j \in Q_1$. Denote by $\repK{Q}{k}$ the category of finite
dimensional $k$-representations of $Q$.

We will use the following notations for a $k$-algebra $\Lambda$ and two finite dimensional
$\Lambda$-modules $M$ and $N$:
\begin{itemize}
  \item $(M,N)_\Lambda := \Hom_\Lambda(M,N)$,
  \item $(M,N)^i_\Lambda := \Ext^i_\Lambda(M,N)$,
  \item $[M,N]_\Lambda := \dim_k\Hom_\Lambda(M,N)$,
  \item $[M,N]^i_\Lambda := \dim_k\Ext^i_\Lambda(M,N)$.
\end{itemize}
Note that $[M,N]^0 = [M,N]$ and $(M,N)^0 = (M,N)$.
If $Q$ is a quiver, we define $(M,N)_Q := (M,N)_{kQ}$ if $M$ and $N$ are $k$-representations and similarly for
the other notations. For dimension vectors $\dvec{d}$ and $\dvec{e}$ we denote
by $\hom_{kQ}(\dvec{d},\dvec{e})$ the minimal value of $[M,N]_Q$ for $M \in \RepK[Q]{\dvec{d}}{k}$ and
$N \in \RepK[Q]{\dvec{e}}{k}$ and similarly for $\ext^i_{kQ}(\dvec{d},\dvec{e})$. More generally, if
$\mathcal{A}$ and $\mathcal{B}$ are subsets of $\RepK[Q]{\dvec{d}}{k}$ and $\RepK[Q]{\dvec{e}}{k}$ respectively, then
$\hom(\mathcal{A}, \mathcal{B})$ and $\ext^i(\mathcal{A}, \mathcal{B})$ are defined analogously.
Whenever the algebra, the field or the quiver is clear from the context, we omit them from the notation.
If our algebra is hereditary we denote $\Ext^1$ by $\Ext$.

Let $M$ be a $k$-representation. A sequence of dimension vectors $\flvec{d}=(\dvec{d}^0, \dots, \dvec{d}^\nu)$
is called a filtration of $M$, or of $\dimve M$, if $\dvec{d}^i \le \dvec{d}^{i+1}$, $\dvec{d}^0 = 0$ and
$\dvec{d}^\nu = \dimve M$.

Let $Q$ be a quiver, $\dvec{d}$ and $\dvec{e}$ dimension vectors and
$\flvec{d} = (\dvec{d}^0, \dots, \dvec{d}^\nu)$ a filtration. By taking
fibre products we can generalise the previous constructions and define
$\Sch{Hom}(\dvec{d}, \dvec{e})$,
$\Sch{Inj}(\dvec{d},\dvec{d}+\dvec{e})$,
$\Sch{Surj}(\dvec{d}+\dvec{e},\dvec{e})$,
$\Gr{\dvec{d}}{\dvec{d}+\dvec{e}}$,
$\OFl(\flvec{d})$ and
$\GL_\dvec{d}$ pointwise.  For all above schemes we can do base change
to any commutative ring $k$, and we will denote these schemes
by e.g. $\Sch{Hom}(\dvec{d}, \dvec{e})_k$.
\begin{definition}
    Define the representation scheme
    \[
    \Rep[Q]{\dvec{d}} := \prod_{\alpha:i \rightarrow j} \Sch{Hom}(d_i,d_j).
    \]
    This is isomorphic to an affine space. Moreover, $\GL_\dvec{d}$ acts
    by base change on $\Rep[Q]{\dvec{d}}$.
\end{definition}
\begin{remark}
   For each scheme defined for a quiver $Q$ we often omit the index if $Q$ is obvious, e.g.
   $\Rep{\dvec{d}} = \Rep[Q]{\dvec{d}}$.
\end{remark}
More generally, we have the module scheme.
\begin{definition}
    Let $k$ be a field, $\Lambda$ a finitely generated $k$-algebra and
    \[\rho \colon k\langle x_1, \dots, x_m\rangle \rightarrow \Lambda\]
    a surjective map from the free associative algebra to $\Lambda$. The affine $k$-scheme $\MMod_\Lambda(d)$ is defined by
    \[ \MMod_\Lambda(d) (R) =
    \Set{ (M^1, \dots, M^m) \in (\End(R^d))^m | \begin{array}{c}f(M^1, \dots, M^m) = 0 \\  \forall\  f \in \Ker \rho\end{array}}.\]
    There is a natural $\GL_d$-action on $\MMod_\Lambda(d)$ given by conjugation.
\end{definition}
\begin{remark}
    We have that $\MMod_\Lambda(d)$ is isomorphic to the functor
    \[
    R \mapsto \Hom_{R-alg}(\Lambda \tensor R, \End(R^d)).
    \]
    Therefore, another choice of $\rho$ gives an isomorphic scheme.
\end{remark}

Let $\{\epsilon_i\}_{i\in I}$, $I$ some index set, be a complete set of (not necessarily primitive) orthogonal idempotents for $\Lambda$.
We say that a $\Lambda$-module $M$ has dimension vector $\dvec{d}=\sum d_i \epsilon_i \in \N^I$ with respect
to this idempotents, if $\dim M_i = \dim M \epsilon_i = d_i$.
The subscheme $\MMod_\Lambda(\dvec{d})$ consisting of $\Lambda$-modules of dimension vector $\dvec{d}$ is an open and closed
subscheme of $\MMod_\Lambda(d)$, therefore
\[ \MMod_\Lambda(d) = \coprod \MMod_\Lambda (\dvec{d}).\]

Let $Q$ be a quiver, $k$ a field, $d$ an integer and $\dvec{d}$ a dimension vector such
that $\sum d_i = d$. Then, after choosing an isomorphism
$\bigoplus k^{d_i} \rightarrow k^d$, there is a natural immersion
\[
\Rep[Q]{\dvec{d}} \rightarrow \MMod_{kQ}(\dvec{d}) \subset \MMod_{kQ}(d).
\]

Note that $\GL_\dvec{d} (k)$-orbits in $\Rep[Q]{\dvec{d}}(k)$ are in one-to-one
correspondence with isomorphism classes of $k$-representations of dimension vector $\dvec{d}$. Similarly,
$\GL_{d}(k)$-orbits in $\MMod_\Lambda(d)(k)$ are in one-to-one correspondence
with isomorphism class of $\Lambda$-modules of $k$-dimension $d$.

We define now quiver flags and quiver Grassmannians as varieties. Choose a $k$-representation
$M \in \Rep[Q]{\dvec{d} + \dvec{e}}(k)$.
We set
\[ \Gr[Q]{\dvec{d}}{M}(R) := \Set{ U \in \Gr{\dvec{d}}{\dvec{d} + \dvec{e}}_k (R) |
\begin{array}{c}
    U \text{ is a subrepesentation of }\\
    M \tensor_k R
\end{array}}
\]
for each $k$-algebra $R$. Then $\Gr[Q]{\dvec{d}}{M}$ is a closed subscheme of $\Gr{\dvec{d}}{\dvec{d}+\dvec{e}}_k$ and
therefore projective. In a similar way we obtain a closed subscheme
$\Fl[Q]{\flvec{d}}{M}$ of $\OFl(\flvec{d})_k$ for a filtration $\flvec{d}$ of $M$.

Let $\Lambda$ be a $k$-algebra and $d$, $e$ two integers.
For a $\Lambda$-module $M \in \MMod_\Lambda(d+e)$
we define
\[ \Gr[\Lambda]{d}{M}(R) := \Set{ U \in \Gr{d}{d + e}(R) | U \text{ is a submodule of }
M \tensor_k R}.
\]

Again, for some set of idempotents $\{\epsilon_i\}$
we define $\Gr[\Lambda]{\dvec{d}}{M}$ as the subscheme consisting of submodules having dimension vector $\dvec{d}$
and obtain
\[ \Gr[\Lambda]{d}{M} = \coprod \Gr[\Lambda]{\dvec{d}}{M},\]
each $\Gr[\Lambda]{\dvec{d}}{M}$ being open and closed in \Gr[\Lambda]{d}{M}.
Moreover, for $\Lambda = kQ$ we obtain that
\[ \Gr[Q]{\dvec{d}}{M} \cong \Gr[kQ]{\dvec{d}}{M} \]
via the immersion $\Rep[Q]{\dvec{d}} \rightarrow \MMod_{kQ}(\dvec{d})$.

We have the following well-known result.
\begin{proposition}
    Let $X, Y$ be two schemes, $Y$ being irreducible.
    Let $f \colon X \rightarrow Y$ be a morphism of schemes such that $f$ is open and for each $z \in Y$ the
    fibre $f^{-1}(z)$ is irreducible. Then $X$ is irreducible.
    \label{isch:propn:openirred}
\end{proposition}
\begin{proof}
    Take $\emptyset \neq U, V \subset X$ and $U, V$ open. We need to show that $U \cap V$ is non-empty. We know that
    $f(U)$ and $f(V)$ are non-empty and open in $Y$. Therefore, they intersect non-trivially. Let
    $y \in f(U) \cap f(V)$.
    By definition, we have
    that $U \cap f^{-1}(y) \neq \emptyset$ and $V \cap f^{-1}(y) \neq \emptyset$ and both are open in $f^{-1}(y)$. By assumption,
    the fibre $f^{-1}(y)$ is irreducible and therefore $U \cap V \cap f^{-1}(y) \neq \emptyset$ and we are done.
\end{proof}
\begin{remark}
    The same is true for arbitrary topological spaces, since the proof relies purely on topology.
\end{remark}

\section{Tangent spaces}
Let $k$ be a field and $\Lambda$ a finitely generated $k$-algebra. We want to calculate the tangent
space at a point of $\Gr[\Lambda]{d}{M}$.
\begin{lemma}
    Let $d, e \in \N$, $k$ a field, $\Lambda$ a finitely generated $k$-algebra and $M \in \MMod_\Lambda(d+e)(k)$.
  Let $U \in \Gr[\Lambda]{d}{M}(K)$, for a field extension $K$ of $k$. Then
  \[ T_U \Gr[\Lambda]{d}{M} \cong \Hom_{\Lambda\tensor_k K} (U, (M\tensor_k K)/U). \]
\end{lemma}
\begin{proof}
  In this proof we use left modules since the notation is more convenient.
  For the tangent space we use the definition of \cite[I, \S 4, no 4]{DG}.
  By base change, we can assume that $K=k$. We prove the claim by doing a $k[\varepsilon]$-valued point calculation ($\varepsilon^2 = 0$).
  The short exact sequence
    \[
    \xymatrix{ 0 \ar[r]& U \ar[r]^\iota & M  \ar[r]^\pi \ar@/^/@{.>}[l]^p
    & M/U \ar[r] \ar@/^/@{.>}[l]^j & 0}
    \]
  is split in the category of $k$-vector spaces, therefore there is a retraction $p$ of $\iota$ and a
  section $j$ of $\pi$. We consider elements of $U$ as elements of $M$ via the inclusion $\iota$.

  The map $k[\varepsilon] \rightarrow k$ given by $s + r \epsilon \mapsto s$ induces
  a map $\theta\colon V \tensor k[\varepsilon] \rightarrow V$ for each $k$-vector space $V$.

  For each homomorphism $f \in \Hom_\Lambda(U, M/U)$ we define
  \[
  S_f := \Set{ u + v \varepsilon | u \in U,\ \pi(v)=f(u) } \subset M \tensor k[\varepsilon].
  \]
  Note that $\theta(S_f) = U$.
  We need to show that $S_f \in \Gr[\Lambda \tensor {k[\varepsilon]}]{d}{M \tensor k[\varepsilon]}$
  and that every element $S$ of the Grassmannian with $\theta(S) = U$ arises in this way.

  First, we show that $S_f$ is a $\Lambda \tensor k[\varepsilon]$-submodule. Let $u + v\varepsilon \in S_f$
  and $r +s\varepsilon \in \Lambda \tensor k[\varepsilon]$. Then we have
  \[ (r+s\varepsilon)(u+v\varepsilon)= ru + (rv + su)\varepsilon. \]
  Since $\pi(v) = f(u)$, $u \in U$ and $f$ is a $\Lambda$-homomorphism we obtain that
  \[ \pi(rv + su) = r\pi(v) = r f(u) = f(ru). \]
  Therefore, $S_f$ is a $\Lambda \tensor k[\varepsilon]$ submodule.

  Now we show that $S_f$ is, as a $k[\varepsilon]$-module, a summand of $M \tensor k[\varepsilon]$ of
  rank $d$. Let $\tilde f\in \Hom_k(U, M):=j \circ f$
  be a $k$-linear lift of $f$.Let
  \begin{alignat*}{2}
      \phi &\colon\quad& U \tensor k[\varepsilon] & \rightarrow M \tensor k[\varepsilon]\\
      && u + v\varepsilon & \mapsto u + (\tilde f (u) +v)\varepsilon.
  \intertext{Obviously, $\phi$ is $k[\varepsilon]$-linear and $\Bild \phi = S_f$. Moreover,
  $\phi$ is split with retraction}
      \psi &\colon\quad& M \tensor k[\varepsilon] & \rightarrow U \tensor k[\varepsilon]\\
      && x + y\varepsilon & \mapsto p(x) + p(y-\tilde f\circ p(x))\varepsilon.
  \end{alignat*}
  Therefore, $S_f$ is a summand of $M \tensor k[\varepsilon]$ of rank $d$.
  
  On the other hand, let $S \in \Gr[\Lambda \tensor {k[\varepsilon]}]{d}{M \tensor k[\varepsilon]}$
  such that $\theta(S)=U$. Then, $U\varepsilon$ is a $k$-subspace of $S$ with
  $\dim_k U\varepsilon = d$, therefore $\dim_k S/(U\varepsilon) = d$. The map sending
  $u + v\varepsilon \in S$ to $u \in U$ is surjective, therefore the induced map from
  $S/(U\varepsilon)$ is an isomorphism.
  Hence, for each $u \in U$ there is a $v \in M$,
  such that $u + v\varepsilon \in S$ and $v$ is unique up to a element in $U\varepsilon$.
  Set $f(u) := \pi(v)$ and, by the discussion before, this does not depend on the choice of $v$
  and we have that $S = S_f$. Moreover,
  $f \in \Hom_{\Lambda}(U,M/U)$: Let $r \in \Lambda$, $u \in U$ and $v \in M$, such that $u+v\varepsilon \in S$.
  Then, $ru + rv\varepsilon \in S$. By definition, $f(ru) = \pi(rv) = r\pi(v) = rf(u)$.
\end{proof}

%
%
\section{Tensor algebras}
We now want to calculate the tangent space at the quiver flag variety by using the previous result on
the tangent space at the module Grassmannian. For this we use tensor algebras, or more precisely
a generalisation of them.

Let $\Lambda_0$ be a ring and $\Lambda_1$ a $\Lambda_0$-bimodule. Define the tensor ring
$T(\Lambda_0, \Lambda_1)$ to be the $\N$-graded $\Lambda_0$-module
\[ \Lambda := \bigoplus_{r\ge 0} \Lambda_r, \quad \Lambda_r := \Lambda_1 \tensor_{\Lambda_0} \dotsm \tensor_{\Lambda_0} \Lambda_1 \: (r \text{ times}),\]
with multiplication given via the natural isomorphism $\Lambda_r \tensor_{\Lambda_0} \Lambda_s \cong \Lambda_{r+s}$. If $\lambda \in \Lambda_r$
is homogeneous, we write $\degr{ \lambda} = r$ for its degree.

The graded radical of $\Lambda$ is the ideal $\Lambda_+ := \bigoplus_{r \ge 1} \Lambda_r$. Note that
$\Lambda_+ \cong \Lambda_1 \tensor_{\Lambda_0} \Lambda$ as right $\Lambda$-modules.

\begin{lemma}
	Let $R$ and $S$ be rings. Let $A_R$, $_R B_S$ and $C_S$ be modules over the corresponding rings. Assume that $_R B_S$ is $S$-projective
	and $R$-flat. Then
	\[ \Ext^n_S (A \tensor_R B, C) \cong \Ext^n_R (A, \Hom_S(B,C)).
	\]
\end{lemma}
\begin{proof}
	Choose a projective resolution $P_\bullet$ of $A_R$. The functor
    $- \tensor_R B$ is exact, and for any projective module $P_R$ the functor
	$\Hom_S(P \tensor_R B, -) \cong \Hom_R(P, -) \circ \Hom_S(B_S, -)$ is exact since $B_S$ is projective. Therefore,
	$P_\bullet \tensor_R B$ gives a projective resolution of $A \tensor_R B_S$ as an $S$-module. We obtain
	\begin{multline*}
		\Ext^n_S (A \tensor_R B, C) = H^n \Hom_S(P_\bullet \tensor_R B, C) \cong \\
		\cong H^n \Hom_R(P_\bullet, \Hom_S(B,C)) = \Ext^n_R (A, \Hom_S(B,C)).
	\end{multline*}
\end{proof}
The following proof is the standard proof for showing that a ``classical'' tensor algebra, i.e.
a tensor algebra where $\Lambda_0$ is semisimple, is hereditary.
\begin{theorem}
	Let $\Lambda$ be a tensor ring and $M \in \MMod \Lambda$.
Then there is a short exact sequence
\[ 0 \rightarrow M \tensor_{\Lambda_0} \Lambda_+ \xrightarrow{\delta_M} M \tensor_{\Lambda_0} \Lambda \xrightarrow{\epsilon_M} M \rightarrow 0,\]
where, for $m \in M$, $\lambda \in \Lambda$ and
$\mu \in \Lambda_1$,
\begin{align*}
	\epsilon_M ( m \tensor \lambda)& := m \cdot \lambda,\\
	\delta_M (m \tensor ( \mu \tensor \lambda) &:= m \tensor (\mu \tensor \lambda) - m \cdot \mu \tensor \lambda.
\end{align*}
Moreover, if $_{\Lambda_0} \Lambda_1$ is flat, then $\pd M \tensor_{\Lambda_0} \Lambda_+ \le \gldim \Lambda_0$ and
$\pd M \tensor_{\Lambda_0} \Lambda \le \gldim \Lambda_0$.
\end{theorem}
\begin{proof}
	It is clear that $\epsilon_M$ is an epimorphism and that $\epsilon_M \delta_M = 0$. To see that $\delta_M$ is a monomorphism, we decompose
	$M \tensor \Lambda = \bigoplus_{r\ge 0} M \tensor \Lambda_r$ and similarly $M \tensor \Lambda_+$ as $\Lambda_0$-modules.
	Then $\delta_M$ restricts to
	maps $M \tensor \Lambda_r \rightarrow (M \tensor \Lambda_r) \oplus (M \tensor \Lambda_{r-1})$ for each $r \ge 1$ and moreover acts as
	the identity on the first component. In particular, if $\sum_{r=1}^t x_r \in \Ker(\delta_M)$ with $x_r \in M \tensor \Lambda_r$, then $x_t =0$.
	Thus $\delta_M$ is injective.

	Next we show that $M \tensor \Lambda = (M \tensor \Lambda_0) \oplus \Bild(\delta_M)$. Let $x = \sum_{r=0}^{t} x_r \in M \tensor \Lambda$.
	We show that $x \in (M \tensor \Lambda_0) + \Bild(\delta_M)$ by induction on $t$. For $t=0$ this is trivial. Let $t \ge 1$. Then
	$x_t \in M \tensor \Lambda_+$ and $x - \delta_M (x_t) = \sum_{r=0}^{t-1} x_r '$. So we are done by induction. If
	$x \in \Bild(\delta_M) \cap M \tensor \Lambda_0$, then there is a $y \in M \tensor \Lambda_+$ such that $\delta_M (y) =x \in M \tensor \Lambda_0$.
	But this is only the case if $y =0$ and therefore $x=0$.

	Now let $_{\Lambda_0} \Lambda_1$ be flat. Then $_{\Lambda_0}\Lambda$ is flat. If $N$ is any $\Lambda_0$-module,
	then $\pd_\Lambda N \tensor \Lambda \le \pd_{\Lambda_0} N$ since
	\[\Ext^n_\Lambda (N \tensor \Lambda, L) \cong \Ext^n_{\Lambda_0} (N, \Hom_{\Lambda}(\Lambda, L)) = \Ext^n_{\Lambda_0} (N, L)\] for any
	$L \in \MMod \Lambda$ by the lemma. Therefore, the second part follows since $\Lambda_+ = \Lambda_1 \tensor \Lambda$.
\end{proof}

Let $R$ be any $k$-algebra of finite global dimension and let $M$, $N$ be two $R$-modules. The Euler form of $M$ with $N$ is defined by
\[
\euf{M,N}_R := \sum_{i=0}^\infty \dim_k \Ext^i(M,N). 
\]
Ringel showed that for a quiver $Q$ and two $k$-representations $M$ and $N$ we have that
\[
\euf{M,N}_Q = \euf{M,N}_{kQ}.
\]

Now let $Q$ be a quiver and $\Lambda_0 = R^{Q_0}$ for a fixed ring $R$. Let $\Lambda_1 := R^{Q_1}$ be the free $\Lambda_0$-bimodule
given by the arrows. The tensor ring $T(\Lambda_0, \Lambda_1)$ is then equal to $RQ$. Note that $\{\epsilon_i\}_{i\in Q_0}$ is
an orthogonal, complete set of idempotents. As before, for an $RQ$-module $M$ we denote
$M\epsilon_i$ by
$M_i$. As $R$-modules we have that $M = \bigoplus M_i$.
\begin{theorem}
	\label{euf_quiverquiver}
	Let $R$ be a $k$-algebra with $\gldim R = n < \infty$. Let $M$ and $N$ be two finite dimensional modules over $RQ$. Then the Euler form is given by
	\[ \euf{M,N}_{RQ} = \sum_{i\in Q_0} \euf{M_i, N_i}_R - \sum_{\alpha\colon i \rightarrow j} \euf{M_i, N_j}_R. \]
\end{theorem}
\begin{proof}
	Let $\Lambda_0 = R^{Q_0}$. There is a natural $\Lambda_0$-bimodule structure on $\Lambda_1 = R^{Q_1}$.
    Let $\Lambda = T(\Lambda_0, \Lambda_1) = RQ$.
	Let $0 \rightarrow M \tensor_{\Lambda_0} \Lambda_1 \tensor_{\Lambda_0} \Lambda \rightarrow M \tensor_{\Lambda_0} \Lambda \rightarrow M \rightarrow 0$
	be the short exact sequence of the previous theorem. Apply $\Hom(-, N)$ to it and consider the long exact sequence
	\[
	\xymatrix@C=12pt{
		0 \ar[r] & \Hom_\Lambda(M, N) \ar[r] & \Hom_\Lambda(M \tensor_{\Lambda_0} \Lambda , N) \ar[r]&
		    \Hom_\Lambda(M \tensor_{\Lambda_0} \Lambda_1 \tensor_{\Lambda_0} \Lambda, N) \ar `r[d] `[lll] `[dlll] `[dll] [dll]\\
		 & \Ext^1_\Lambda(M, N) \ar[r] &\Ext^1_\Lambda(M \tensor_{\Lambda_0} \Lambda , N) \ar[r]&
		    \Ext^1_\Lambda(M \tensor_{\Lambda_0} \Lambda_1 \tensor_{\Lambda_0} \Lambda, N) \\
			&\cdots & &  \\
		 \ar[r]& \Ext^n_\Lambda(M, N) \ar[r] &\Ext^n_\Lambda(M \tensor_{\Lambda_0} \Lambda , N) \ar[r]&
		    \Ext^n_\Lambda(M \tensor_{\Lambda_0} \Lambda_1 \tensor_{\Lambda_0} \Lambda, N) \\
		\ar[r]& \Ext^{n+1}_\Lambda(M, N) \ar[r] & **[l]0.\qquad &}
	\]
    Here we obtain the last $0$ since the $\pd M \tensor \Lambda \le \gldim R$
    by the previous theorem.

    Now, since $ _{\Lambda_0} \Lambda_\Lambda$ is projective as a $\Lambda_0$-module and as a $\Lambda$-module, we obtain
	\[\Ext^i_\Lambda(M \tensor_{\Lambda_0} \Lambda , N) \cong \Ext^i_{\Lambda_0}(M,N) \cong \bigoplus_{i \in Q_0} \Ext^i_R(M_i, N_i)\]
	and
	\[\Ext^i_{\Lambda}(M \tensor_{\Lambda_0} \Lambda_1 \tensor_{\Lambda_0} \Lambda, N) \cong \Ext^i_{\Lambda_0} (M \tensor_{\Lambda_0} \Lambda_1, N) \cong
	\bigoplus _{\alpha \colon i \rightarrow j} \Ext^i_R(M_i, N_j). \]
	This yields the claim.
\end{proof}
\begin{remark}
    Note that, for a field $k$, we recover Ringel's result, namely that
    \[
    \sum_{i\in Q_0} \dim M_i \dim N_i - \sum_{\alpha\colon i \rightarrow j} \dim M_i \dim N_j = [M,N]_{kQ} - [M,N]^1_{kQ}
    \]
    for two $k$-representations $M$ and $N$ of a quiver $Q$.
\end{remark}
\section{Geometry of Quiver Flags}
Now we come back to quiver flags. Let $k$ be a field.
We can consider $\Rep[Q]{\dvec{d}}$ as an affine scheme over $k$ with the obvious functor of points. More
generally, we work in the category of schemes over $k$.
Fix a filtration
$\seqv{\dvec{d}}=(\dvec{d}^0 = 0\le \dvec{d}^1 \le \dots \le \dvec{d}^\nu)$.

Denote by $A_{\nu+1}$ the quiver having $\nu+1$ vertices numbered from $0$ to $\nu$ and arrows going from vertex $i$ to vertex $i+1$
for $0 \le i < \nu$. Let $\Lambda := (kQ) A_{\nu+1}$. Then $\Mod \Lambda$ is the category of sequences of $k$-representations of
$Q$ of length $\nu+1$ and
chain maps between them, i.e. a morphism between two modules
\[
\boldsymbol{M}= M^0 \rightarrow M^1 \rightarrow \dots \rightarrow M^\nu
\]
and
\[
\boldsymbol{N}= N^0 \rightarrow N^1 \rightarrow \dots \rightarrow N^\nu
\]
is given by a commutative diagram
\[
\begin{CD}
    M^0 @>>> M^1 @>>> \cdots @>>> M^\nu\\
     @VVV   @VVV    @.          @VVV\\
    N^0 @>>> N^1 @>>> \cdots @>>> N^\nu.
\end{CD}
\]
For $\Lambda$ it is easy to calculate the Euler form of two modules $\seqv{M}, \seqv{N} \in \Mod \Lambda$ by using
theorem $\ref{euf_quiverquiver}$:
\[ \euf{\seqv{M},\seqv{N}}_\Lambda = \sum_{i=0}^r \euf{M^i,N^i}_{kQ} - \sum_{i=0}^{r-1} \euf{M^i, N^{i+1}}_{kQ}. \]

We show that $\Gr[\Lambda]{\seqv{\dvec{d}}}{\seqv{M}} \cong \Fl[Q]{\seqv{\dvec{d}}}{M}$, where
$\seqv{M} = (M=M=\dots=M)$, and then use the previous
results to calculate the tangent space.
\begin{lemma}
    Let $\flvec{d}$ be a filtration and
    $M \in \Rep[Q]{\dvec{d}^\nu}(k)$.
    Let
    \[
    \seqv{U} = (U^0, U^1, \dots, U^\nu) \in \Fl[Q]{\seqv{\dvec{d}}}{M}(K)
    \]
    for a field extension
    $K$ of $k$. Then we have that
    \[
T_\seqv{U} \Fl[Q]{\seqv{\dvec{d}}}{M} \cong \Hom_{\Lambda\tensor K}(\seqv{U}, (\seqv{M}\tensor K)/\seqv{U}),
    \]
    where $\Lambda = (kQ)A_{\nu + 1}$ and $\seqv{M} = (M=M=\dots=M)
    \in \MMod_\Lambda (\dvec{d}^\nu,\dvec{d}^\nu, \dots, \dvec{d}^\nu)(k)$.
\end{lemma}
\begin{proof}
    For a submodule
    \[
    \seqv{U}=(U^0 \rightarrow U^1 \rightarrow\dots\rightarrow U^\nu) \in \Gr[\Lambda]{\seqv{\dvec{d}}}{\seqv{M}}(R)
    \]
    we have automatically that the maps $U^i \rightarrow U^{i+1}$ are injections. Therefore,
    such a submodule $\seqv{U}$ gives, in a natural way, rise to a flag
    $\seqv{U} \in \Fl[Q]{\seqv{\dvec{d}}}{M}(R)$ and vice versa. This yields an isomorphism
    $\Gr[\Lambda]{\seqv{\dvec{d}}}{\seqv{M}} \cong \Fl[Q]{\seqv{\dvec{d}}}{M}$. Since
    $\Gr[\Lambda]{\seqv{\dvec{d}}}{\seqv{M}}$ is open in
    $\Gr[\Lambda]{d}{\seqv{M}}$, where $d = \sum_{k,i} d^k_i$, we have that, for a point
    $\seqv{U} \in \Fl[Q]{\seqv{\dvec{d}}}{M}(K)$,
    \[ T_\seqv{U} \Fl[Q]{\seqv{\dvec{d}}}{M} \cong T_\seqv{U} \Gr[\Lambda]{\seqv{\dvec{d}}}{\seqv{M}}
    = T_\seqv{U} \Gr[\Lambda]{d}{\seqv{M}} = \Hom_{\Lambda\tensor K}(\seqv{U}, (\seqv{M}\tensor K)/\seqv{U}). \]
\end{proof}


We define the closed subscheme $\RepFl[Q]{\seqv{\dvec{d}}}$ of
$\Rep[Q]{\dvec{d}^\nu} \times \OFl(\seqv{\dvec{d}})$ by its functor of points
\[ \RepFl[Q]{\seqv{\dvec{d}}}(R) := \Set{ (M, \seqv{U}) \in
\Rep[Q]{\dvec{d}^\nu}(R) \times \OFl(\flvec{d})(R) | \seqv{U} \in \Fl[Q]{\seqv{\dvec{d}}}{M}}. \]

We have the following.

\begin{lemma}
    Let $\flvec{d}$ be a filtration.
Consider the two natural projections from the fibre product restricted to $\RepFl[Q]{\seqv{\dvec{d}}}$.
\[
\xymatrix{
\RepFl[Q]{\seqv{\dvec{d}}} \ar[r]^{\pi_1} \ar[d]_{\pi_2}& \Rep[Q]{\dvec{d}^\nu}\\
\OFl(\flvec{d})&
}
\]
    \label{lmm:geomflags:projtriv}
Then $\pi_1$ is projective and $\pi_2$ is a vector bundle of rank
\[
\sum_{k=1}^\nu \sum_{\alpha:i \rightarrow j} d^k_j (d_i^k - d_i^{k-1}).
\]
Therefore, $\RepFl[Q]{\seqv{\dvec{d}}}$ is smooth and irreducible of dimension
\[ \sum_{k=1}^{\nu-1} \euf{\dvec{d}^k, \dvec{d}^{k+1} - \dvec{d}^k}_Q + \dim \Rep[Q]{\dvec{d}^\nu}. \]
Finally, the (scheme-theoretic) image
$\mathcal{A}_\flvec{d}:=\pi_1(\RepFl[Q]{\seqv{\dvec{d}}})$ is a closed, irreducible
subvariety of $\Rep[Q]{\dvec{d}^\nu}$.
\end{lemma}
\newcommand{\myunit}{1.5ex}
\tikzset{%
    node style ge/.style={rectangle,minimum size=\myunit}
}
\begin{proof}
    $\pi_1$ is projective since it factors as a closed immersion into projective space times
    $\ORep_{Q}$ followed by the projection to $\ORep_{Q}$.

    For $\dvec{I}=(I_i)_{i \in Q_0}$, each $I_i \subset \{1,\dots,d^\nu_i\}$,
    we set $W_\dvec{I}$ to be the
    graded subspace
    of $k^{\dvec{d}^\nu}$ with basis $\{e_j\}_{j \in I_i}$ in the $i$-th graded part $k^{d^\nu_i}$ and
    $|\dvec{I}| := (|I_i|)_{i \in Q_0} \in \N^{Q_0}$.
    For a sequence $\seqv{\dvec{I}} = (\dvec{I}^0, \dvec{I}^1, \dots, \dvec{I}^\nu)$ such that
    $I_i^k \supset I_i^{k+1}$ and $|\dvec{I}^k|= \dvec{d}^\nu - \dvec{d}^k$ we set
    \[
    \seqv{W}_{\seqv{\dvec{I}}} := (W_{\dvec{I}^0}, \dots, W_{\dvec{I}^\nu})
    \]
    to be the decreasing sequence of
    subspaces associated to $\seqv{\dvec{I}}$.
    We show that $\pi_2$ is trivial over the open affine subset $U_{\seqv{\dvec{I}}}$ of
    $\OFl(\seqv{\dvec{d}})$
    given by
    \[
    U_{\seqv{\dvec{I}}} (R) := \Set{ \seqv{U} \in \OFl(\flvec{d})(R) |
    U^k \oplus (W_{\dvec{I}^k} \tensor R) = R^{\dvec{d}^\nu}}.
    \]
    Without loss
    of generality we assume $I^k_i = \{d^k_i+1, \dots, d^\nu_i\}$. Each element
    $\seqv{U} \in U_\seqv{\dvec{I}}(R)$ is given uniquely by some matrices
    $A_i^k \in \Mat_{(d^\nu_i-d^k_i)\times (d^k_i-d^k_{i-1})}$
    such that
 \[U^k_i = \Bild
\begin{array}{c@{}}\begin{tikzpicture}
\matrix (A) [matrix of math nodes,%
             nodes = {node style ge},%
             left delimiter  = (,%
             right delimiter = )]
             {%
         \id_{d^1_i} & 0 & \cdots & 0\\
        \ & \id_{d^2_i-d^1_i} & \ddots & \vdots\\
        \ & \hphantom{AB} & \ddots &0\\
        \ & A^2_i & \ddots & \id_{d^{k}_i-d^{k-1}_i}\\
        \hphantom{AB} & \hphantom{AB} & \cdots & A^k_i\\
         };
         \path ($ (A-2-1)!0.5!(A-4-1) $) node {$A^1_i$};
         \draw (A-1-1.south west) rectangle (A-5-1.south east);
         \draw ($ (A-2-2.south west)!0.3!(A-2-2.south) $) rectangle (A-5-2.south east);
         \draw (A-5-4.north west) rectangle (A-5-4.south east);
     \end{tikzpicture}\end{array}.\]

      Let $V^k:=W_{(\{1,\dots,d^k_i\})}$. Let $X$ be the closed subscheme of
      $\Rep[Q]{\dvec{d}^\nu}$
      given by the functor of points
      \[X(R) := \Set{ M \in \Rep[Q]{\dvec{d}^\nu} | V^k \tensor R \text{ is a subrepresentation of } M 
      \; \forall \; 0 \le k \le \nu}.\]
      Note that $X$ is an affine space of dimension
      \[\sum_{k=1}^\nu \sum_{\alpha:i \rightarrow j} d^k_j (d_i^k - d_i^{k-1}).\]
      Let $g_\seqv{U}:= (g_i)_{i\in Q_0}$ where
    \[g_i =
\begin{array}{c@{}}
\begin{tikzpicture}
\matrix (A) [matrix of math nodes,%
             nodes = {node style ge},%
             left delimiter  = (,%
             right delimiter = )]
             {%
      \id_{d^1_i} & 0 & \cdots & 0 & 0\\
      \hphantom{AB} & \id_{d^2_i-d^1_i} & \ddots & \vdots & 0\\
       & \hphantom{AB}& \ddots & 0 & \vdots\\
      \hphantom{AB} & A^2_i & \ddots & \id_{d^{\nu-1}_i-d^{\nu-2}_i} & 0\\
      \hphantom{AB} &\hphantom{AB}  & \cdots & A^{\nu-1}_i & \id_{d^{\nu}_i-d^{\nu-1}_i}\\
      };
         \path ($ (A-2-1)!0.5!(A-4-1) $) node {$A^1_i$};
         \draw (A-1-1.south west) rectangle (A-5-1.south east);
         \draw ($ (A-2-2.south west)!0.3!(A-2-2.south) $) rectangle (A-5-2.south east);
         \draw (A-5-4.north west) rectangle (A-5-4.south east);
  \end{tikzpicture}
  \end{array}
      \in \GL_{d^\nu_i}(k).\]
      Then, the map from $X \times U_{\seqv{\dvec{I}}}$ to
      $U_{\seqv{\dvec{I}}} \times_{\OFl} \RepFl[Q]{\flvec{d}}$ given by sending
      $(M, \seqv{U})$ to $(g_\seqv{U} \cdot M, \seqv{U})$ is an isomorphism which induces
      an isomorphism of vector spaces on the fibres.
      Therefore, we have that $\pi_2$ is a vector bundle.

      Finally, we prove the claim on dimension. Since $\OFl(\flvec{d})$ is smooth,
      we have that
      \begin{multline*}
          \dim \OFl(\flvec{d}) =
          \sum_{i \in Q_0} \sum_{k=1}^{\nu-1} \sum_{l=k+1}^\nu (d^k_i - d^{k-1}_{i}) (d^l_i - d^{l-1}_{i})\\
          =\sum_{i \in Q_0} \left( \sum_{k=1}^{\nu-1} \sum_{l=k+1}^\nu d^k_i (d^l_i - d^{l-1}_{i}) -
          \sum_{k=1}^{\nu-2} \sum_{l=k+2}^\nu d^{k}_{i} (d^l_i - d^{l-1}_{i}) \right)\\
          =\sum_{i \in Q_0} \left( d^{\nu-1} (d^\nu_i - d^{\nu-1}_i) + \sum_{k=1}^{\nu-2} d^k_i (d^{k+1}_i - d^k_i) \right) =
          \sum_{i \in Q_0} \left(\sum_{k=1}^{\nu-1} d^k_i (d^{k+1}_i - d^k_i) \right).
      \end{multline*}
      Since
      $\RepFl[Q]{\flvec{d}}$ is smooth and $\pi_2$ is a vector bundle we obtain
      \begin{multline*}
        \dim \RepFl[Q]{\flvec{d}} = \dim \OFl(\flvec{d}) +
        \sum_{\alpha \colon i \rightarrow j}\sum_{k=1}^\nu (d^k_i- d^{k-1}_{i}) d^{k}_j\\
          =\sum_{i \in Q_0} \left(\sum_{k=1}^{\nu-1} d^k_i (d^{k+1}_i - d^k_i) \right) +
          \sum_{\alpha \colon i \rightarrow j}\sum_{k=1}^\nu (d^k_i- d^{k-1}_{i}) d^{k}_j\\
          =\sum_{k=1}^{\nu-1} \left( \sum_{i \in Q_0}d^k_i (d^{k+1}_i - d^k_i)  +
          \sum_{\alpha \colon i \rightarrow j} d^k_i (d^{k}_j-d^{k+1}_j)\right)
          + \sum_{\alpha \colon i \rightarrow j} d^\nu_i d^\nu_j\\
      =\sum_{k=1}^{\nu -1} \euf{\dvec{d}^k, \dvec{d}^{k+1} - \dvec{d}^k}_Q + \dim \Rep[Q]{\dvec{d}^\nu}.
  \end{multline*}
\end{proof}
\begin{remark}
    Note that if $M$ is a $k$-valued point of $\mathcal{A}_\flvec{d}$ for $k$ not algebraically closed,
    then $M$ does not necessarily have a flag of type $\flvec{d}$. This only becomes true after a finite
    field extension.
\end{remark}
We now can give an estimate for the codimension of $\mathcal{A}_\flvec{d}$ in $\Rep[Q]{\dvec{d}^\nu}$.
For this we use Chevalley's theorem.
\begin{theorem}[Chevalley]
    Let $k$ be a field, $X, Y$ irreducible $k$-schemes and $f \colon X \rightarrow Y$ a dominant morphism.
    Then for every point $y \in Y$ and every point $x \in f^{-1}(y)$, the scheme theoretic fibre, we
    have that
    \[ \dim_x f^{-1}(y) \ge \dim X - \dim Y. \]
    Moreover, on an open, non-empty subset of $X$ we have equality.
\end{theorem}
\begin{proof}
    See \cite[\S 5, Proposition 5.6.5]{EGA4}.
\end{proof}
\begin{theorem}
  Let $\flvec{d}$ be a filtration, $K$ a field extension of $k$, $\Lambda := (KQ)A_{\nu+1}$ and
  $(M,\seqv{U}) \in \RepFl[Q]{\flvec{d}}(K)$. Let
  $\seqv{M}=(M=\dots=M)$ as a $\Lambda$-module. Then we have that
  \[ \codim \mathcal{A}_\flvec{d} \le \dim \Ext^1_{\Lambda}(\seqv{U}, \seqv{M}/\seqv{U})
  \le \dim \Ext^1_{\Lambda}(\seqv{U}, \seqv{M})
  \le \dim \Ext^1_{KQ}(M, M). \]
  \label{them:codimflag}
\end{theorem}
\begin{proof}
  Since $\dim \mathcal{A}_{\flvec{d}}$ is stable under flat base change we can
  assume $k=K$.
  Let $\seqv{V} := \seqv{M}/\seqv{U}$.
  Then we have the following short exact sequence of $\Lambda$-modules:
  \[
  \begin{CD}
    0:@.\quad           @.  0   @>>>  0  @>>> \cdots @>>> 0\\
    @VVV @.  @VVV      @VVV         @.       @VVV\\
    \seqv{U}:@.\quad   @.  U^0 @>>> U^1 @>>> \cdots @>>> U^\nu\\
    @VVV @.  @VVV      @VVV         @.       @VVV\\
    \seqv{M}:@.\quad   @.   M @=     M  @=   \cdots @=    M\\
    @VVV @.  @VVV      @VVV         @.       @VVV\\
    \seqv{V}:@.\quad   @.  V^0 @>>> V^1 @>>> \cdots @>>> V^\nu\\
    @VVV @.  @VVV      @VVV         @.       @VVV\\
    0:@.\quad         @.    0   @>>>  0  @>>> \cdots @>>> 0.\\
  \end{CD}
  \]
  We already know that $\Hom_\Lambda(\seqv{U}, \seqv{V})$
  is the tangent space of $\Fl{\flvec{d}}{M}$ at the point $\seqv{U}$.
  Using Chevalley's theorem, we have that
  \[\dim \Hom_\Lambda(\seqv{U}, \seqv{V}) \ge \dim_\seqv{U}\Fl{\flvec{d}}{M} \ge
  \dim \RepFl[Q]{\flvec{d}} - \dim \mathcal{A}_\flvec{d} \]
  and therefore
  \[ \dim \mathcal{A}_\flvec{d} \ge \dim \RepFl[Q]{\flvec{d}} - \dim \Hom_\Lambda(\seqv{U}, \seqv{V}).\]
  We now calculate
  \begin{multline*}
    \euf{\seqv{U}, \seqv{V}}_\Lambda = \sum_{k=0}^\nu \euf{U^k, V^k}_{Q} - \sum_{k=0}^{\nu-1} \euf{U^k, V^{k+1}}_{Q}\\
    = \sum_{k=1}^{\nu-1} \euf{\dvec{d}^k, \dvec{d}^\nu - \dvec{d}^k}_{Q} - \sum_{k=1}^{\nu-1}
    \euf{\dvec{d}^k, \dvec{d}^\nu - \dvec{d}^{k+1}}_{Q}
    = \sum_{k=1}^{\nu-1} \euf{\dvec{d}^k, \dvec{d}^{k+1} - \dvec{d}^k}_{Q}.
  \end{multline*}
  Recall that
\[ \dim \RepFl[Q]{\seqv{\dvec{d}}} = \sum_{k=1}^{\nu-1}
\euf{\dvec{d}^k, \dvec{d}^{k+1} - \dvec{d}^k}_Q + \dim \Rep[Q]{\dvec{d}^\nu}. \]
  In total
  \begin{multline*}
    \codim \mathcal{A}_\flvec{d} \le \dim \Rep{\dvec{d}} + \dim \Hom_\Lambda (\seqv{U}, \seqv{V})
    - \dim \RepFl[Q]{\flvec{d}}\\
    = \dim \Hom_\Lambda (\seqv{U}, \seqv{V}) - \sum_{k=1}^{\nu-1} \euf{\dvec{d}^k, \dvec{d}^{k+1} - \dvec{d}^k}_{Q}
    =\dim \Hom_\Lambda (\seqv{U}, \seqv{V}) - \bform{\seqv{U}}{\seqv{V}}_\Lambda\\
    = \dim \Ext^1_\Lambda(\seqv{U}, \seqv{V}) - \dim \Ext^2_\Lambda(\seqv{U}, \seqv{V}).
  \end{multline*}
  Here we have the last equality since $\gldim \Lambda \le 2$.
  Since $\gldim kQ = 1$ and $P = P = \dots = P$ is projective in $\Mod \Lambda$
  for every projective $P$ in $\Mod kQ$ we see that
  $\pd_\Lambda \seqv{M} \le 1$ and similarly $\id_\Lambda \seqv{M} \le 1$. Consider, as before, the short exact sequence
  \[ 0 \rightarrow \seqv{U} \rightarrow \seqv{M} \rightarrow \seqv{V} \rightarrow 0. \]
  Since $\seqv{M}$ has projective dimension less than two we have that
  $(\seqv{U}, \seqv{U})^2=0$ and that $(\seqv{U}, \seqv{V})^2 = 0$.
  Applying $(-,\seqv{M})$ gives
  a surjection $(\seqv{M},\seqv{M})^1 \rightarrow (\seqv{U}, \seqv{M})^1$.
  Applying $(\seqv{U},-)$ yields a surjection $(\seqv{U}, \seqv{M})^1 \rightarrow (\seqv{U}, \seqv{V})^1$.
  Hence the above result simplifies to
  \[ \codim \mathcal{A}_\flvec{a} \le \dim \Ext^1(\seqv{U}, \seqv{V}) \le \dim \Ext^1(\seqv{U}, \seqv{M})
  \le \dim \Ext^1(\seqv{M}, \seqv{M}). \]

  Obviously, $\Ext^1_\Lambda(\seqv{M}, \seqv{M}) \cong \Ext^1_{kQ}(M,M)$ and the claim follows.
\end{proof}
\begin{remark}
    Note that if the characteristic of $k$ is $0$, then, by generic smoothness,
    there is a point $M \in \mathcal{A}_\flvec{d}$ and an $\seqv{U} \in \Fl{\flvec{d}}{M}$ such
    that $\Fl{\flvec{d}}{M}$ is smooth in $\seqv{U}$ and
    the value of $\dim \Ext^1_{\Lambda \tensor K}(\seqv{U}, (M\tensor K)/\seqv{U})$ is minimal. In this case we have that
  \[ \codim \mathcal{A}_\flvec{d} = \dim \Ext^1_{\Lambda \tensor K}(\seqv{U}, (M\tensor K)/\seqv{U}).\]
\end{remark}

  We also construct an additional vector bundle.
  \begin{definition}
      Let $\flvec{d}$ be a filtration.
      Let $\Sch{Rep}_{Q,A_{\nu+1}}(\flvec{d})$ be the scheme given via its functor of points
      \begin{multline*}
          \Sch{Rep}_{Q,A_{\nu+1}}(\flvec{d})(R) := \\
          \Set{ (\seqv{U}, \seqv{f})  \in \prod_{i=0}^\nu \Rep[Q]{\dvec{d}^i}(R) \times
          \prod_{i=0}^{\nu-1} \Hom(\dvec{d}^i, \dvec{d}^{i+1})(R) | f^i \in \Hom_{RQ}(U^i, U^{i+1})}.
      \end{multline*}
      Let $\Sch{IRep}_{Q,A_{\nu+1}}(\flvec{d})$ be the open subscheme of
      $\Sch{Rep}_{Q,A_{\nu+1}}(\flvec{d})$ given by its functor of points
      \[
      \Sch{IRep}_{Q,A_{\nu+1}}(\flvec{d})(R) := \Set{ (\seqv{U}, \seqv{f}) \in \Sch{Rep}_{Q,A_{\nu+1}}(\flvec{d})(R) |
      f^i \in \Sch{Inj}(\dvec{d}^i, \dvec{d}^{i+1})(R)}.
      \]
  \end{definition}
  \begin{remark}
      Note that $\Sch{Rep}_{Q, A_{\nu +1}}(\flvec{d})(k)$ consists of sequences
      of $k$-representations of $Q$. Therefore, these are modules over $(kQ)A_{\nu+1}$. Vice
      versa, every $(kQ)A_{\nu+1}$-module of dimension vector $\flvec{d}$ is isomorphic to an element of
      $\Sch{Rep}_{Q, A_{\nu +1}}(\flvec{d})(k)$.

      We will often write $\seqv{U}$ instead of $(\seqv{U}, \seqv{f})$ for
      an $(\seqv{U}, \seqv{f}) \in \Sch{Rep}_{Q, A_{\nu +1}}(\flvec{d})(R)$.
  \end{remark}
  \begin{lemma}
      Let $\flvec{d}$ be a filtration.
      Then the projection
      \begin{alignat*}{2}
          \pi&\colon\quad& \Sch{IRep}_{Q,A_{\nu+1}}(\flvec{d}) &\rightarrow
          \prod_{i=0}^{\nu-1} \Sch{Inj}(\dvec{d}^i, \dvec{d}^{i+1})\\
          \intertext{given by sending}
      &&(\seqv{U}, \seqv{f}) &\mapsto \seqv{f}
      \end{alignat*}
      is a vector bundle and therefore flat.
      In particular, $\Sch{IRep}_{Q,A_{\nu+1}}(\flvec{d})$ is smooth and irreducible.
  \end{lemma}
  \begin{proof}
      The first part is analogously to lemma \ref{lmm:geomflags:projtriv}. Irreducibility
      then follows by the fact that flat morphisms are open and proposition \ref{isch:propn:openirred}.
  \end{proof}
  Before we continue, we give the following easy lemma, stated by K. Bongartz
  in \cite{Bongartz_singularities}, which gives
  rise to a whole class of vector bundles.
  \begin{lemma}
      \label{lmm:geomflags:vecbun}
      Let $X$ be a variety over a ground ring $k$. Let $m, n \in \N$ and
      $f \colon X \rightarrow \Hom(m,n)_k$ a morphism. Then
      for any $r\in \N$, the variety $X(r)$ given by the functor of points
      \[ X(r)(R):= \Set{ x \in X(R) | f(x) \in\Hom(m,n)_{m-r}(R)} \]
      is a locally closed subvariety of $X$. Moreover, the closed subvariety
      \[
      U_r(R):=\Set{ (x, v) \in X(r)(R)\times R^m | f(x)(v) = 0 }
      \]
      of $X(r) \times k^m$
      is a sub
      vector bundle
      of rank $r$ over $X(r)$.
  \end{lemma}
  \begin{proof}
  The claim follows easily by using Fitting ideals.
  \end{proof}
  \begin{example}
      Let $(\seqv{M},\seqv{g}) \in \Rep[Q,A_{\nu+1}]{\flvec{e}}(k)$.
      
      Let $\varphi \colon \Rep[Q,A_{\nu+1}]{\flvec{d}} \rightarrow \Hom(m,n)$
      be the morphism given by
      \[ (\seqv{U},\seqv{f}) \mapsto \left( \seqv{h}=(h_i^k) \mapsto
      \left( (h^k_j U_\alpha^k - M_\alpha^k h^k_i)_{\substack{\alpha\colon i \rightarrow j\\0 \le k \le \nu}}, 
      (h^{k+1}_i f^k_i - g^k_i h^k_i)_{\substack{i \in Q_0\\0 \le k \le \nu}} \right) \right)
      \]
      for every $(\seqv{U}, \seqv{f}) \in \Rep[Q,A_{\nu+1}]{\flvec{d}}(R)$,
      where
      \begin{align*}
          m &= \sum_{i \in Q_0} \sum_{k=0}^{\nu} d_i^k e_i^k &
          &\text{and}&
      n &= \sum_{k=0}^\nu \sum_{\alpha \colon i \rightarrow j} d_i^k e_j^k + 
      \sum_{k=0}^{\nu-1} \sum_{i \in Q_0} d_i^k e_i^{k+1}.
      \end{align*}
      Then $\seqv{h} \in \ker \varphi (\seqv{U}, \seqv{f})$ if and only if
      $\seqv{h} \in \Hom_{(R Q)A_{\nu+1}}(\seqv{U}, \seqv{M}\tensor R)$.

      Set
      \[
      \RepHom[Q,A_{\nu+1}]{\flvec{d}, \seqv{M}}_r := U_r(R)
      \]
      from the previous lemma.
      Note that elements $(\seqv{U}, \seqv{h}) \in \RepHom[Q,A_{\nu+1}]{\flvec{d}, \seqv{M}}_r(R)$
      are all pairs consisting of a representation $\seqv{U} \in \Rep[Q,A_{\nu+1}]{\flvec{d}}(R)$
      and a morphism $\seqv{h} \in \Hom_{(RQ)A_{\nu+1}}(\seqv{U}, \seqv{M})$ such that
      $\rank \Hom_{(RQ)A_{\nu+1}}(\seqv{U}, \seqv{M})=r$.

      The lemma yields that the projection
      \begin{align*}
          \RepHom[Q,A_{\nu+1}]{\flvec{d}, \seqv{M}}_r &\rightarrow \Rep[Q,A_{\nu+1}]{\flvec{d}}(r)\\
          (\seqv{U}, \seqv{h}) &\mapsto \seqv{U}
      \end{align*}
      is a vector bundle of rank $r$.

      It also stays a vector bundle if we restrict it to
      the open subset $\Sch{IRep}_{Q,A_{\nu+1}}(\flvec{d})(r)$ of $\Rep[Q,A_{\nu+1}]{\flvec{d}}(r)$.
      We denote the preimage under the projection to this variety by
      $\Sch{IRepHom}_{Q, A_{\nu+1}}(\flvec{d},\seqv{M})_r$.\label{ex:geomflags:hombundle}
%
  \end{example}
  We obtain the following.
  \begin{theorem}
      Let $\flvec{d}$ be a filtration and $K$ a field extension of $k$.
      \label{theom:geomflags:irred}
      Assume that there is an $M \in \mathcal{A}_\flvec{d}(K)$ such that $\dim \Ext^1_{KQ}(M,M) =
      \codim \mathcal{A}_\flvec{d}$. Then $\Fl[Q]{\flvec{d}}{M}$ is smooth over $K$ and
      geometrically irreducible.
  \end{theorem}
  \begin{proof}
      Smoothness is an immediate consequence of the last theorem, since we have that
      $\dim T_{\seqv{U}}\Fl[Q]{\flvec{d}}{M}$ is constant and smaller or equal
      to the dimension at each irreducible component living in $\seqv{U}$. Since $K$
      is a field this implies smoothness. See \cite[I, \S 4, no 4]{DG}.

      Now we prove irreducibility. By base change we can assume that $K$ is algebraically closed.
      Consider all the following schemes as $K$-varieties.
      We construct the following diagram of varieties.
\[
\xymatrix{ & \Sch{IRepHom}_{Q, A_{\nu+1}}(\flvec{d},\seqv{M})_r \ar@{->>}[d]^{\text{vector bundle}} &
\Sch{IRepInj}_{Q, A_{\nu +1}}(\flvec{d},\seqv{M})_r \ar@{_{(}->}[l]_{\text{open}} \ar@{->>}[d]\\
\Sch{IRep}_{Q,A_{\nu+1}}(\flvec{d}) & \ar@{_{(}->}[l]_{\text{open}} \Sch{IRep}_{Q,A_{\nu+1}}(\flvec{d})(r)
& \Fl[Q]{\flvec{d}}{M},\\
}
\]
$r$ being equal to $\euf{\flvec{d}, \seqv{M}}_\Lambda + [M,M]^1$.
Since open subvarieties and images of irreducible varieties are again irreducible and by application
of proposition \ref{isch:propn:openirred} we then obtain that $\Fl[Q]{\flvec{d}}{M}$ also is irreducible.

  Consider the minimal value $r$ of $\dim\Hom(\seqv{U}, \seqv{M})$ for
  $\seqv{U} \in \Sch{IRep}(\flvec{d})(K)$. Denote by
  \begin{alignat*}{2}
      \pi \colon&& \Sch{IRep}(\flvec{d}) &\rightarrow \Rep[Q]{\dvec{d}^\nu}\\
                && \seqv{U} &\mapsto U^\nu.
  \end{alignat*}
  Since $\Orbit_M$ is open in $\mathcal{A}_\flvec{d}$ and $\Sch{IRep}$ is irreducible,
  the intersection of the two open sets $\pi^{-1}(\Orbit_M)$ 
  and
  $\Sch{IRep}(r)$ is non-empty. For all elements $\seqv{U}$ of $\pi^{-1}(\Orbit_M)$
  we have, by theorem \ref{them:codimflag}, that $[\seqv{U}, \seqv{M}]^1_\Lambda = [M,M]^1_Q$.
  We already saw that $[\seqv{U}, \seqv{M}]^2 = 0$, therefore
  \[
  [\seqv{U}, \seqv{M}]_\Lambda = \euf{\seqv{U}, \seqv{M}}_\Lambda + [\seqv{U}, \seqv{M}]^1_\Lambda
  = \euf{\flvec{d}, \seqv{M}}_\Lambda + [M,M]^1_Q.
  \]
  This means that the dimension of the homomorphism space is constant on $\pi^{-1}(\Orbit_M)$ and
  we obtain that $r=\euf{\flvec{d}, \seqv{M}}_\Lambda + [M,M]^1_Q$. Moreover, $\Sch{IRep}_r$
  is irreducible as an open subset of $\Sch{IRep}$.

  We then have that $\Sch{IRepHom}_{Q,A_{\nu+1}}(\flvec{d},\seqv{M})_r$
  is irreducible, since it is a vector bundle on $\Sch{IRep}_r$ by example \ref{ex:geomflags:hombundle}.
  Take the open subvariety $\Sch{IRepInj}
  (\flvec{d},\seqv{M})_r$
  of $\Sch{IRepHom}
  (\flvec{d},\seqv{M})_r$ where the morphism to $\seqv{M}$ is injective.
  It is irreducible as an open subset of an irreducible variety.
  The projection
  from this variety to $\Fl[Q]{\flvec{d}}{M}$ is surjective since
  $\pi^{-1}(\Orbit_M)$ is contained in $\Sch{IRep}(\flvec{d})(r)$, and therefore $\Fl[Q]{\flvec{d}}{M}$ is irreducible.
  \end{proof}

  We now want to interpret theorem \ref{theom:geomflags:irred} in terms of Hall numbers. Let $X_0$
  be a variety defined over a finite field $\F_q$, where $q=p^n$ for a prime $n$. Denote
  by $\overline{\F_q}$ the algebraic closure of $\F_q$ and by $X := X_0 \tensor \overline{\F_q}$
  the variety obtained from $X_0$
  by base change from $\F_q$ to $\overline{\F_q}$. Let $F$ be the Frobenius automorphism acting on $X$.
  Denote by $H^i(X, \Q_\ell)$ the $\ell$-adic cohomology group with compact support for a prime $\ell \neq p$,
  see for example \cite{Freitag_etaleweil}. Denote by $F^i$ the induced action of $F$ on
  cohomology $H^i(X, \Q_\ell)$.
  P. Deligne proved the following theorem.
\begin{theorem}[P. Deligne \cite{Deligne_weil2}, 3.3.9]
Let $X_0$ be a proper and smooth variety over $\F_q$.
For every $i$, the characteristic polynomial $\det( T \id - F^i, H^i(X, \Q_\ell))$
is a polynomial with coefficients in $\Z$, independent of $\ell$ ( $\ell \neq p)$.
The complex roots $\alpha$ of this polynomial have absolute value $\lvert \alpha \rvert = q^{\frac{i}{2}}$.
\end{theorem}

Moreover, the Lefschetz fixed point formula yields that
\[
\#X_0(\F_{q^n}) = \sum_{i \ge 0} (-1)^i \Tr( (F^i)^n, H^i(X, \Q_\ell)).
\]

Assume now that there is a polynomial $P \in \Q[t]$ such that, for each finite field extension
$L/\F_q$ we have that $\# X_0(L) = P(|L|)$, $|L|$ being the number of elements of the
finite field. We call $P$ the counting polynomial of $X_0$. Then we have the following.
\begin{theorem}
	Let $X_0$ be proper and smooth over $\F_q$ with counting polynomial $P$.
	Then odd cohomology of $X$ vanishes and
    \[ P(t) = \sum_{i=0}^{\dim X_0} \dim H^{2i}(X, \Q_\ell) t^i.
	\]
\end{theorem}
\begin{proof}
	See \cite[Lemma A.1]{BillVandenBergh_absolutelyindec}.
\end{proof}

Assume now that $Y$ is a projective scheme over $\Z$ and set
$Y_k:=Y \tensor k$ for any field $k$. Note that for $Y$ $\ell$-adic
cohomology agrees with $\ell$-adic cohomology with compact support.
Assume furthermore that there is a counting polynomial $P \in \Q[t]$ such that,
for each finite field $k$, we have that $\# Y_k(k) = P(|k|)$.
By the base change theorem \cite[Theorem 1.6.1]{Freitag_etaleweil} we have
\[ H^i(Y_{\overline{\Q}}, \Q_\ell) \cong H^i (Y_\C, \Q_\ell). \]
By the comparison theorem \cite[Theorem 1.11.6]{Freitag_etaleweil} we have
\[ H^i(Y_\C, \Q_\ell) \cong H^i (Y_\C(\C), \Q_\ell),\]
where on the right hand side we consider the usual cohomology of the complex analytic manifold
attached to $Y_\C$.

Moreover, there is an open, dense subset $U$ of $\Spec \Z$ such that
$H^i (Y_{\overline{\kappa(v)}}, \Q_\ell) \cong H^i (Y_{\overline{\Q}}, \Q_\ell)$ for all
$v \in U$, where $\kappa(v)$ denotes the residue field at $v$. This means that for almost all primes $p$ we have that
\[H^i (Y_{\overline{\F_p}}, \Q_\ell) \cong H^i (Y_{\overline{\Q}}, \Q_\ell) \cong H^i(Y_\C(\C), \Q_\ell).\]

Therefore, if we know the Betti numbers of $Y_\C(\C)$, then we know the coefficients
of the counting polynomial. In order to apply this to our situation we use
the following theorem of W. Crawley-Boevey \cite{Bill_rigidintegral}.
\begin{theorem}
    \label{theom:geomflags:rigidreps}
	Let $M$ be an $k$-representation without self-extensions. Then there is
	a $\Z$-representation $N$ such that $M= N \tensor k$ and for
	all fields $K$ we have that $\Ext (N \tensor K, N \tensor K) = 0$.
\end{theorem}
Putting all this together, we obtain the following.
\begin{theorem}
      Assume that there is an $M \in \mathcal{A}_\flvec{d}(\Q)$, being a direct sum of
      exceptional representations, such that
      \[
      \dim \Ext^1_{\Q Q}(M,M) =
      \codim \mathcal{A}_\flvec{d}.
      \]
      Let $N$ be a $\Z$-representation and $P \in \Q[t]$ a polynomial, such that
      $N \tensor \Q \cong M$ and
      $\# \Fl[Q]{\flvec{d}}{N \tensor \F_q} = P(q)$ for every prime power $q$. Then
      $P(0) = 1$ and $P(1) = \chi\left( \Fl[Q]{\flvec{d}}{M \tensor \C}\right) > 0$.

      Moreover, if $Q$ is Dynkin or extended Dynkin, then there is a representation $N$ and
      a polynomial $P$ with the required properties.
\end{theorem}
\begin{proof}
  Let $X := \Fl[Q]{\flvec{d}}{N}$ as a scheme
  over $\Z$.
  Using theorem \ref{theom:geomflags:irred} we
  obtain that $X_k$ is smooth and irreducible for every field $k$. By the previous discussion
  we have then that the
  $i$-th coefficient of $P$ is exactly
  $\dim H^{2i}(X_\C(\C), \Q_\ell)$ and that odd cohomology vanishes. Therefore,
  \[
  0<\sum\dim H^{2i}(X_\C(\C), \Q_\ell) = \chi\left( X_\C \right) = P(1).
  \]
  By irreducibility we have that
  \[
  P(0) = \dim H^{0}(X_\C(\C), \Q_\ell) = 1
  \]
  and this proves the first claims.

  If $Q$ is Dynkin or extended Dynkin,
  then let $N$ be the $\Z$-representation given by theorem \ref{theom:geomflags:rigidreps}. We have the
  polynomial $P$ since we have Hall polynomials in both cases for exceptional representations
  by results of Ringel \cite{Ringel_hallpolysforrepfiniteheralgs} and Hubery
  \cite{Hubery_hallpolys}.
  \end{proof}

Let $q$ be a prime power and $M$ an $\F_q$-representation of
a Dynkin quiver $Q$. If $\overline{\Orbit_M}=\mathcal{A}_w$ for some word $w$ in simples, then
$F_w^M = 1 \mod q$. Therefore, in the Dynkin case, the generic coefficient
in the expression $u_w$ in the Hall algebra is equal to one modulo $q$. In section \ref{sec:reflhallnum}
we will show that even
more is true, namely that every non-zero coefficient in this expression is equal to one modulo $q$.

\section{A geometric version of BGP reflection functors}
In this chapter we want to define a geometric version of BGP reflection functors. First
we recall the definition of the usual BGP reflection functors as introduced by 
Bern{\v{s}}te{\u\i}n, Gel$'$fand and Ponomarev
\cite{BernsteinGelfandPonomarev_coxeter}. Then we go on to define reflection functors on flags
of subrepresentations. In the remainder of this paper we fix a field $k$ and consider only
$k$-valued points of all varieties, if not otherwise stated.

Now we recall the definition of BGP reflection functors. The
main reference for this section is \cite{Ringel_integral}. For a nice
introduction see \cite{HKrause_repsofalgs}.

Let $Q$ be a quiver and $\dvec{d}$ be a dimension vector.
For each vertex $a \in Q_0$ we have the reflection
\begin{alignat*}{2}
 \sigma_a &\colon\quad & \Z Q_0 & \rightarrow \Z Q_0\\
 & & \dvec{d} &\mapsto \dvec{d} - \sbform{\dvec{d}}{\epsilon_a}_Q \epsilon_a.
\end{alignat*}
If $Q$ has no loop at $a$, one easily checks that $\sigma_a^2 \dvec{d} = \dvec{d}$.

We also define reflections on the quiver itself. The quiver $\sigma_a Q$ is
obtained from $Q$ by reversing all arrows ending or starting in $a$. If $\alpha\colon a \rightarrow j$
is an arrow in $Q_1$, then we call $\alpha^* : j \rightarrow a$ the arrow in the other direction and
analogously for $\alpha \colon i \rightarrow a$. We have that $(\sigma_a Q)_0 = Q_0$, therefore we can
regard $\sigma_a \dvec{d}$ as a dimension vector on $\sigma_a Q$.
Obviously, $\sigma_a^2 Q =Q$ if we identify $\alpha^{* *}$ with $\alpha$ (and we will do
this in the remainder).

If $a$ is a sink of $Q$, we define for each $k$-representation $M$ of $Q$ the homomorphism
\[
\phi^M_a \colon \bigoplus_{\alpha: j \rightarrow a} M_j \xrightarrow{\left(
\begin{smallmatrix}
	M_\alpha
\end{smallmatrix}\right)} M_a.
\]
Dually, if $b$ is a source of $Q$, we define
\[
\phi^M_b \colon M_b \xrightarrow{\left(
\begin{smallmatrix}
	    M_\alpha
	\end{smallmatrix}\right)} \bigoplus_{\alpha: b \rightarrow j} M_j .
\]
Note that $D\phi^{DM}_a = \phi^M_a$, where $D$ denotes the $k$-dual, and
that $d_a - \rank \phi^X_a = \dim \Hom(X, S_a)$ for $a$ a sink of $Q$ and
a representation $X$ of dimension vector $\dvec{d}$.

We define a pair of reflection functors
$S_a^+$ and $S_b^-$. To this end we fix a $k$-representation $M$ of $Q$ of dimension vector $\dvec{d}$.
If the vertex $a$ is a sink of $Q$ we construct
\[ S_a^+ \colon \repK{Q}{k} \rightarrow \repK{\sigma_a Q}{k} \]
as follows. We define $S_a^+ M := N$ by letting $N_i := M_i$ for a vertex $i \neq a$ and letting $N_a$ be
the kernel of the map $\Phi^M_a$. Denote by
$\iota \colon \Ker \phi^M_a \rightarrow \bigoplus M_j$ the canonical inclusion and by
$\pi_i \colon \bigoplus M_j \rightarrow M_i$ the canonical projection.
Then, for each
$\alpha: i \rightarrow a$ we let $N_{\alpha^*} \colon \Ker \phi^M_a \rightarrow M_i$
be
the composition $\pi_i \circ \iota$ making the following diagram commute
\[
\begin{CD}
	0 @>>>  \Ker \phi^M_a  @>>\iota> \bigoplus M_j @>>> M_a\\
	@.     @VV{N_{\alpha^*}}V  @VV{\pi_i}V @.\\
	@.    M_i @=       M_i. \\
\end{CD}
\]
We obtain
a $k$-representation $N$ of $\sigma_a Q$. We call this representation $S^+_a M$.
This construction yields a functor $S^+_a$.

Let $s := d_a - \rank \phi^M_a$ be
the codimension
of $\Bild \phi^M_a$ in $M_a$.
We have that
\[e_a := \dim (S^+_a M)_a = \sum_{\alpha: i \rightarrow a} d_i - \rank \Phi^M_a  =
d_a - (\dvec{d}, \epsilon_a) + s = d_a -s  - (\dvec{d}- s \epsilon_a, \epsilon_a).\]
Therefore, $\dimv{ S^+_a M } = \sigma_a (\dvec{d} - s \epsilon_a) = \sigma_a (\dvec{d}) + s \epsilon_a$.
Dually, for a sink of $b$ of $Q$ we
obtain a functor $S^-_b$.

For a sink $a$ of $Q$ we have that $(S_a^-, S_a^+)$ is a pair of adjoint functors and
that $S_a^+$ is left exact and $S_a^-$ is right exact. There is a natural
monomorphism $\iota_{a,M} \colon S_a^- S_a^+ M \rightarrow M$ for $M \in \repK{Q}{k}$ and a natural epimorphism
$\pi_{a,N} \colon N \rightarrow S_a^+ S_a^- N$ for $N \in \repK{\sigma_a Q}{k}$.
We have the following lemma.
\begin{lemma}
 Let $a$ be a sink and $X$ an indecomposable $k$-representation of $Q$. Then, the following are equivalent:
 \begin{enumerate}
 \item $X \ncong S_a$.
 \item $S^+_a X$ is indecomposable.
 \item $S^+_a X \neq 0$.
 \item $S_a^- S_a^+ X \cong X$ via the natural inclusion.
 \item The map $\Phi_a^X$ is an epimorphism.
 \item $\sigma_a (\dvec{\dim} X) > 0$.
 \item $\dvec{\dim} S^+_a X = \sigma_a (\dvec{\dim} X)$.
 \end{enumerate}
\end{lemma}

An sequence $(i_1, \dots, i_r)$ of the vertices of $Q$ is called admissible,
if
$i_p$ is a sink in $\sigma_{i_{p-1}} \dotsm \sigma_{i_1} Q$ for each $1 \le p \le r$.
An admissible sequence is called admissible ordering if each vertex of $Q$ appears exactly once
in the sequence.
Note that there is an admissible ordering if and only if
$Q$ has no oriented cycles. Such a quiver is called acyclic. In the following we always assume that $Q$ is
acyclic.
For any admissible sequence of sinks $w=(i_1, \dots, i_r)$ define
\[S_w^+ := S_{i_r}^+ \circ \dots \circ S_{i_1}^+. \]

If $w=(i_1, \dots, i_n)$ is an admissible ordering of $Q$, we define the Coxeter functors
\begin{align*}
  C^+ &:= S^+_{i_n} \circ \dots \circ S^+_{i_1} & C^- &:= S^-_{i_1} \circ \dots \circ S^-_{i_n}. \\
\end{align*}
Since both reverse each arrow of $Q$ exactly twice, these are endofunctors of $\repK{Q}{k}$. Neither functor
depends on the
choice of the admissible ordering.

A $k$-representation $P$ is projective if and only if $C^+ P =0$. Dually, a $k$-representation $I$
is injective if and only if $C^- I =0$.
  An indecomposable $k$-representation $M$ of $Q$ is called preprojective if $(C^+)^r M=0$ for some $r\ge0$
  and preinjective if $(C^-)^r M =0$ for some $r\ge0$. If $M$ is neither preprojective nor preinjective, then
  $M$ is called regular.

  An arbitrary $k$-representation is called preprojective (or preinjective or regular) if it is isomorphic to a direct sum of indecomposable
  preprojective (or preinjective or regular) representations.

Let $M$ be a $k$-representation of $Q$ and $\flvec{d}$ a filtration of $\dimve M$.
We want to define reflections on a flag $\seqv{U} \in \Fl[Q]{\seqv{\dvec{d}}}{M}$. Let $a$ be
a sink and let $\seqv{U} \in \Fl[Q]{\seqv{\dvec{d}}}{M}$. Then, for each $i$, we have the following commutative diagram with exact rows.
\[
\begin{CD}
    0 @>>>  (S^+_a U^{i-1})_a      @>>> \bigoplus U^{i-1}_j @>{\phi^{U^{i-1}}_a}>> \Bild \phi^{U^{i-1}}_a @>>> 0\\
    @.          @VVfV                        @VVV           @VVgV     @.\\
  0 @>>> (S^+_a U^i)_a @>>> \bigoplus U^i_j @>>\phi^{U^i}_a>
  \Bild \phi^{U^i}_a @>>> 0\\
\end{CD}
\]
By definition, $g$ and the map in the middle are injective. Therefore, $f$ is injective. This immediately yields that
$S^+_a \seqv{U}$ is a new quiver flag of $S^+_a M = S^+_a U^\nu$. The problem is that the
dimension of $S^+_a U^i$ is dependent on the rank of $\phi^{U^i}_a$.
This motivates the next definition. Recall that $d_a - \rank \phi^X_a = \dim \Hom(X, S_a)$
for a representation $X$ of dimension vector $\dvec{d}$.
\begin{definition}
  Let $a$ be a sink, $\dvec{d}$ a dimension vector and $s$ an integer. Define
    \[ \Rep[Q]{\dvec{d}}\lrang{a}^{s} :=
    \Set{ M \in \Rep[Q]{\dvec{d}} | \dim \Hom(M, S_a) = s}.\]

    Let $\seqv{\dvec{d}}$ be a $(\nu+1)$-tuple of dimension vectors
    and
    $\seqv{r} = (r^0, r^1, \dots , r^\nu)$ be a $(\nu+1)$-tuple of integers.
    For each representation $M$ define
\[
\Fl[Q]{\seqv{\dvec{d}}}{M}\lrang{a}^{\seqv{r}} :=
\Set{ \seqv{U} \in \Fl[Q]{\seqv{\dvec{d}}}{M} |
    \dim \Hom(U^i, S_a) = r^i}.
\]

     Moreover, let
     \[ \Rep[Q]{\dvec{d}}\lrang{a} := \Rep[Q]{\dvec{d}}\lrang{a}^0\] and
     \[ \Fl[Q]{\seqv{\dvec{d}}}{M}\lrang{a} := \Fl[Q]{\seqv{\dvec{d}}}{M}\lrang{a}^{\seqv{0}}.\]
\end{definition}
\begin{remark}
    Recall that a filtration $\flvec{d}$ of $\dimve M$ is a sequence of dimension
    vectors such that $\dvec{d}^0=0$, $\dvec{d}^\nu = \dimve M$ and that $\dvec{d}^i \le \dvec{d}^{i+1}$.
  In order to know a filtration $\flvec{d}$ it is enough to know the terms
  $\dvec{d}^1, \dots, \dvec{d}^{\nu-1}$, since $\dvec{d}^0$ is always $0$ and
  $\dvec{d}^\nu$ is always $\dimve M$. Therefore, we identify the $(\nu-1)$-tuple
  $(\dvec{d}^1, \dots, \dvec{d}^{\nu-1})$ with the $(\nu+1)$-tuple $\flvec{d}$.
\end{remark}
\begin{example}
    Let
\[Q =
\makeatletter%
\let\ASYencoding\f@encoding%
\let\ASYfamily\f@family%
\let\ASYseries\f@series%
\let\ASYshape\f@shape%
\makeatother%
\setlength{\unitlength}{1pt}
\includegraphics{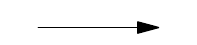}%
\definecolor{ASYcolor}{gray}{0.000000}\color{ASYcolor}
\fontsize{12.000000}{14.400000}\selectfont
\usefont{\ASYencoding}{\ASYfamily}{\ASYseries}{\ASYshape}%
\ASYalign(-53.064667,4.770035)(-0.500000,-0.500000){1.000000 0.000000 0.000000 1.000000}{$1$}
\ASYalign(-3.840845,4.770035)(-0.500000,-0.500000){1.000000 0.000000 0.000000 1.000000}{$2$}
\ASYalign(-28.452756,8.383535)(-0.500000,0.000000){1.000000 0.000000 0.000000 1.000000}{$\scriptstyle\alpha$}
.\]
Consider the representation $M$ given by $M_1 = M_2 = k^2$ and
$M_\alpha =\left(\begin{smallmatrix}
1 & 0\\
0 & 0\\
\end{smallmatrix}\right)$. We have that
$M \in \Rep[Q]{(2,2)}\lrang{2}^1$. Now consider flags of type $((0,0), (1,1), (2,2))$, i.e.
subrepresentations $N$ of dimension vector $(1,1)$. We need two injective linear maps
$f_1, f_2 \colon k^1 \rightarrow k^2$ making the following diagram commutative.
\[
\xymatrix{
k \ar[r]^{N_\alpha} \ar[d]_{f_1} & k \ar[d]^{f_2}\\
k^2 \ar[r]_{\left(\begin{smallmatrix}
    1 & 0\\
    0 & 0
\end{smallmatrix}\right)} & k^2
}
\]
We have the following situations.
\label{bsp:grass}
\begin{itemize}
    \item $N \in \Gr[Q]{(1,1)}{M}\lrang{2}^1$: This means that $N_\alpha = 0$. Therefore,
        we need that the image of $f_1$ is in
        the kernel of $M_\alpha$, which is $1$-dimensional. Hence, a subrepresentation in
        $\Gr[Q]{(1,1)}{M}\lrang{2}^1$ is given by $f_1 = \left( \begin{smallmatrix}
            0\\
            1
        \end{smallmatrix}\right)$ and $f_2$ being an arbitrary inclusion. The point
        $f_1 = f_2 = \left( \begin{smallmatrix}
            0\\
            1
        \end{smallmatrix}\right)$ is special, since for this inclusion we have
        that $M/N \cong S_1 \oplus S_2$ and otherwise $M/N \cong k \overset{1}{\rightarrow} k$.
    \item $N \in \Gr[Q]{(1,1)}{M}\lrang{2}^0$: This means that $N_\alpha \neq 0$. Therefore,
        we need that the image of $f_1$ is not in
        the kernel of $M_\alpha$, which is $1$-dimensional. Hence, a subrepresentation in
        $\Gr[Q]{(1,1)}{M}\lrang{2}^0$ is given by $f_1=f_2= \left( \begin{smallmatrix}
            1\\
            x
        \end{smallmatrix}\right)$ for any $x \in k$.
\end{itemize}
The variety $\Gr[Q]{(1,1)}{M}$ consists therefore of two $\Proj^1_k$ glued together at one point.
Graphically,
\[
\Gr[Q]{(1,1)}{M} = {\color{red}\Gr[Q]{(1,1)}{M}\lrang{2}^1} \amalg {\color{blue}\Gr[Q]{(1,1)}{M}\lrang{2}^0} \cong
\begin{array}{c@{}}
    \begin{tikzpicture}
        \draw[blue] (1,0) circle (.5);
        \draw[red] (0,0) circle (.5);
        \filldraw[fill=red,red] (0.5,0) circle (0.025); 
    \end{tikzpicture}
\end{array}.
\]
Note that the Grassmannian is neither irreducible nor smooth.
\end{example}

In order to get rid of $\seqv{r}$ we define the following maps and then look at the fibres.
\begin{definition}
    Let $a$ be a sink, $\dvec{d}$ a dimension vector and $M \in \Rep[Q]{\dvec{d}}\lrang{a}^s$ for an $s\in \N$.
    We have that $M \cong M' \oplus S_a^s$ for some element $M' \in \Rep[Q]{\dvec{d} - s\epsilon_a}\lrang{a}$.
    Without loss of generality we can assume that $M=M' \oplus S_a^s$ and we set
    $\pi_a M := M'$. Obviously, $\pi_a M$ is unique up to isomorphism.

  Now let $\flvec{d}$ be a filtration
    and $\seqv{r}=(r^0, \dots, r^\nu)$ a $(\nu +1)$-tuple of integers.
    Define
    \begin{alignat*}{3}
    \pi_a^{\seqv{r}} &\colon\quad&  \Fl[Q]{\seqv{\dvec{d}}}{M}\lrang{a}^{\seqv{r}} &\rightarrow
      \Fl[Q]{\seqv{\dvec{d}} - \seqv{r} \epsilon_a}{\pi_a M}\lrang{a}\\
      && \seqv{U} &\mapsto \seqv{V} & & \text{where } V^i_j:=\begin{cases}
      U^i_j & \text{ if } j \neq a,\\
      \Bild \phi^{U^i}_a & \text{ if } j = a.
    \end{cases}
    \end{alignat*}
    \label{def:flagtoredfib}
\end{definition}
\begin{remark}
    Note that $\pi_a M$ is $\iota_{a,M} S^-_a S^+_a M$.
\end{remark}
\begin{example}
    Coming back to example \ref{bsp:grass} we see that $\pi_2^1$ collapses $\Gr[Q]{(1,1)}{M}\lrang{2}^1$
    to the point
    \[\Gr[Q]{(1,0)}{k^2 \overset{\left(\begin{smallmatrix}
        1 & 0\\
    \end{smallmatrix}\right)}{\rightarrow} k}\lrang{2}.\]
    The fibre of $\pi_2^1$ over this point is the vector space Grassmannian $\Gr{1}{k^2}$, being
    isomorphic to $\Proj^1_k$.
\end{example}

We now introduce a little bit more notation. If $\seqv{d}$ is a sequence, then denote by
$\overleftarrow{\seqv{d}}$ the sequence given by $(\overleftarrow{\seqv{d}})^i = d^{\nu-i}$. Moreover,
we define the sequence $\seqv{e}$ by $e^i := d^\nu - d^{\nu-i}$.
Therefore, if $\seqv{\dvec{d}}$ is a filtration of $\dvec{d}^\nu$,
then $\seqv{\dvec{e}}$ is a filtration of $\dvec{d}^\nu$.

The fibre of the map $\pi_a^{\seqv{r}}$ is a set of the following type.
\begin{definition}
  Let $\seqv{e} = (e^0, e^1, \dots, e^\nu)$ and $\seqv{r} = (r^0, r^1, \dots , r^\nu)$ be sequences of non-negative integers such that
  $\seqv{e} + \overleftarrow{\seqv{r}}$ is a filtration.
  Let $A_{\nu+1}$ be the quiver
  \[
  0 \rightarrow 1 \rightarrow 2 \rightarrow 3 \rightarrow \dotsm \rightarrow \nu.
  \]
  Then define
  \[
  \xymatrix{
  X^{\seqv{r}, \seqv{e}} := &
  k^{e^{\nu}+r^0} \ar@{->>}[r] &k^{e^{\nu-1}+r^1} \ar@{->>}[r] & \cdots \ar@{->>}[r] & k^{e^1 + r^{\nu-1}}\ar@{->>}[r]  & k^{e^0 + r^\nu}.
  }
  \]
\end{definition}
\begin{remark}
    The $k$-representation $X^{\seqv{r}, \seqv{e}} \in \repK{A_{\nu +1}}{k}$ is injective and its
    isomorphism class does not depend on the choice of the surjections.
\end{remark}
\begin{lemma}
  Let $\seqv{e} = (e^0, e^1, \dots, e^\nu)$
  and $\seqv{r} = (r^0, r^1, \dots , r^\nu)$ be sequences of non-negative integers such that
  $\seqv{e} + \overleftarrow{\seqv{r}}$ is a filtration.
  Then $X^{\seqv{r}, \seqv{e}}$ has
  a subrepresentation of dimension vector $\seqv{r}$ if and only if $\seqv{e}$ is a filtration of $e^\nu$
  \label{lmm:flagfib_not_empty}. Moreover, if $k$ is a finite field with $q$ elements, then
  the number of $k$-subrepresentations is given by
  \[
  \# \Gr[A_{\nu+1}]{\seqv{r}}{X^{\seqv{r},\seqv{e}}} = \prod_{i=0}^\nu \qbinom{e^{\nu-i} - e^{\nu-i-1} + r^i}{r^i}_q.
  \]
  In particular, this number is equal to $1$ modulo $q$ if and only if the set of subrepresentations is non-empty.
\end{lemma}
\begin{proof}
  We prove this by induction on $\nu$.
\begin{description}
\item[$\nu=0$]
  There is a subspace of dimension $r^0$ of $k^{r^0 + e^0}$ if and only if $e^0 \ge 0$ and, for
  $k$ a finite field of cardinality $q$, the number of
  those is obviously $\qbinom{e^0 + r^0}{r^0}_q$.
\item[$\nu \ge 1$]
    If $(U^0,U^1,U^2, \dots, U^\nu)$ is a subrepresentation of dimension vector $\seqv{r}$
    of $X^{\seqv{r}, \seqv{e}}$,
then
$(U^1,U^2, \dots, U^\nu)$ is a subrepresentation of dimension vector $(r^1, r^2,\dots, r^\nu)$ of
$X^{(r^1, r^2, \dots, r^\nu), (e^0, e^1, \dots, e^{\nu-1})}$. Therefore, by induction,
$e^i \le e^{i+1}$ for $0 \le i < \nu-1$ and $0 \le e^{0}$. The preimage $V$ of $U^1$ under the surjection from $U^0$ has
dimension $r^1 + ( (e^{\nu} + r^0) - (e^{\nu-1} + r^1))= e^\nu - e^{\nu-1} + r^0$. Since $\seqv{U}$ is a subrepresentation, we must have that
$U^0 \subset V$. Therefore, $r^0 \le e^\nu - e^{\nu-1} + r^0$ or equivalently $e^{\nu -1} \le e^\nu$.

On the other hand, if $e^0 \ge 0$ and $e^i \le e^{i+1}$ for all $0 \le i < \nu$, then there is a subrepresentation
$(U^1,U^2, \dots, U^\nu)$ of $X^{(r^1, r^2, \dots, r^\nu), (e^0, e^1, \dots, e^{\nu-1})}$ of dimension vector $(r^1, r^2,\dots, r^\nu)$ by induction.
As before, the dimension of the preimage $V$ of $U^1$ under the surjection from $U^0$ has dimension
$e^\nu - e^{\nu-1} + r^0 \ge r^0$. If we choose now any subspace $U^0$ of dimension $r^0$ in $V$, then we obtain
a subrepresentation
of $X^{\seqv{r}, \seqv{e}}$ of dimension vector $\seqv{r}$.

If $k$ is a finite field of cardinality $q$, then, by induction, we have that
the number of subrepresentations
of $X^{(r^1, r^2, \dots, r^\nu), (e^0, e^1, \dots, e^{\nu-1})}$
of dimension vector $(r^1, r^2,\dots, r^\nu)$ is equal to
  \[
  \prod_{i=1}^\nu \qbinom{e^{\nu-i} - e^{\nu-i-1} + r^i}{r^i}_q.
  \]
  To complete such a subrepresentation to a subrepresentation of $X^{\seqv{r}, \seqv{e}}$ we
  have to choose an $r^0$-dimensional subspace of an $(e^{\nu} - e^{\nu-1} + r^0)$-dimensional
  space. Therefore, the number of subrepresentations is equal to
  \[
  \qbinom{e^{\nu} - e^{\nu-1} + r^0}{r^0}_q \prod_{i=1}^\nu \qbinom{e^{\nu-i} - e^{\nu-i-1} + r^i}{r^i}_q.
  \]
  This yields the claim.
  \end{description}
\end{proof}

\begin{theorem}
    Let $a$ be a sink of $Q$, $\flvec{d}$ a filtration, $\seqv{r}$ a $(\nu+1)$-tuple of non-negative integers
    and $M \in \Rep[Q]{\dvec{d}^\nu}$. Then
    \[\pi_a^{\seqv{r}} \colon
      \Fl[Q]{\seqv{\dvec{d}}}{M}\lrang{a}^{\seqv{r}} \rightarrow
      \Fl[Q]{\seqv{\dvec{d}} - \seqv{r} \epsilon_a}{\pi_a M}\lrang{a}\]
      is surjective and
      the fibre $(\pi_a^{\seqv{r}})^{-1}(\seqv{U})$
      over any $\seqv{U} \in \Fl[Q]{\seqv{\dvec{d}} - \seqv{r} \epsilon_a}{\pi_a M}\lrang{a}$
      is isomorphic to $\Gr[A_\nu]{\seqv{r}}{X^{\seqv{r}, \seqv{\dvec{e}}_a}}$,
      where $\dvec{e}^{\nu -i} := \dvec{d}^\nu - \dvec{d}^i$.
      In particular the number
      of points in the fibre only depends on $\seqv{r}$ and $\seqv{\dvec{d}}$ and not on $\seqv{U}$.
     \label{theom:qgrass_to_red_fib}
\end{theorem}
\begin{proof}
  Fix a flag $\seqv{V} \in \Fl[Q]{\seqv{\dvec{d}} - \seqv{r} \epsilon_a}{\pi_a M}\lrang{a}$. Let
  $\seqv{U} \in \Fl[Q]{\seqv{\dvec{d}}}{M}\lrang{a}^{\seqv{r}}$. The flag $\seqv{U}$ is
  in $(\pi_a^{\seqv{r}})^{-1}(\seqv{V})$ if and only if $U_j^i = V_j^i$ for all $j \neq a$
  in which case $\Bild(\Phi^{U^i}_a) = V^i_a$.
  Therefore, we only have to choose $U^i_a \subset M_a$ such that $V^i_a \subseteq U^i_a$,
  $U^{i-1}_a \subset U^i_a$ and $\dim U^i_a = d^i_a$ for all $i$. This is the same as choosing
  $\overline{U^{i}_a} \subset M_a/V^i_a$ such that $\theta^i(\overline{U^{i-1}_a}) \subset \overline{U^i_a}$ and
  $\dim \overline{U^i_a} = r^i$ if we denote by $\theta^i \colon M_a/V^{i-1}_a \rightarrow M_a/V^i_a$ the
  canonical projection. This is equivalent to finding a subrepresentation of
  \[ M_a/V^0_a \rightarrow M_a/V^1_a \rightarrow \cdots \rightarrow M_a/V^\nu_a \]
  of dimension vector $\seqv{r}$. All the maps in this representation are surjective since $\seqv{V}$
  is a flag, therefore this representation
  of $A_{\nu+1}$ is isomorphic to
  $X^{\seqv{r}, \seqv{e}}$. Since $\flvec{d}$ is a filtration of $M$ we have that $\flvec{e}$ is a filtration.
  Therefore, $\Gr[A_\nu]{\seqv{r}}{X^{\seqv{r}, \seqv{\dvec{e}}_a}}$ is non-empty by lemma \ref{lmm:flagfib_not_empty}
  and $\pi_a^\seqv{r}$ is surjective.
\end{proof}
Now we are nearly ready to do reflections. The only thing left to define is what happens on a source. If $b$ is a source in
$Q$, then $b$ is a sink in $Q^{op}$, so we just dualise everything. Denote by $D:=\Hom_k(-,k)$.
\begin{definition}
  Let $\seqv{U} \in \Fl[Q]{\seqv{\dvec{d}}}{M}$ and let $\dvec{e}^{\nu-i} = \dvec{d}^\nu - \dvec{d}^i$.
  Then define
  \begin{alignat*}{2}
    \hat D &\colon\quad &
    \Fl[Q]{\seqv{\dvec{d}}}{M}&\rightarrow
    \Fl[Q^{op}]{\seqv{\dvec{e}}}{DM}\\
    &&\seqv{U} & \mapsto (\hat D (\seqv{U}))^i := \ker( DM \rightarrow D(U^i)) = D(M/U^i).
  \end{alignat*}
\end{definition}
\begin{remark}
    Obviously, $\hat D ^2 \cong \id$ and the map $\hat D$ is an isomorphism of varieties.
\end{remark}
\begin{definition}
  Let $b$ be a source, $\dvec{d}$ a dimension vector and $s$ an integer. Define
    \[ \Rep[Q]{\dvec{d}}\lrang{b}^{s} :=
    \Set{ M \in \Rep[Q]{\dvec{d}} | \dim \Hom(S_b, M) = s}.\]

    Let $\seqv{\dvec{d}}$ be a $(\nu+1)$-tuple of dimension vectors
    and
    $\seqv{r} = (r^0, r^1, \dots , r^\nu)$ be a $(\nu+1)$-tuple of integers.
    For each representation $M$ define
\[
\Fl[Q]{\seqv{\dvec{d}}}{M}\lrang{b}^{\seqv{r}} :=
\Set{ \seqv{U} \in \Fl[Q]{\seqv{\dvec{d}}}{M} |
    \dim \Hom(S_b, M/U^i) = r^i}.
\]

     Moreover, let
     \[ \Rep[Q]{\dvec{d}}\lrang{b} := \Rep[Q]{\dvec{d}}\lrang{b}^0\] and
     \[ \Fl[Q]{\seqv{\dvec{d}}}{M}\lrang{b} := \Fl[Q]{\seqv{\dvec{d}}}{M}\lrang{b}^{\seqv{0}}.\]
\end{definition}
%
%
\begin{remark}
    Note that $\seqv{U} \in \Fl[Q]{\seqv{\dvec{d}}}{M}\lrang{b}^{\seqv{r}}$ if and only if
    $\hat D \seqv{U} \in \Fl[Q^{op}]{\seqv{\dvec{e}}}{DM}\lrang{b}^{\seqv{r}}$.
\end{remark}
Now we state the main result on reflections.
\begin{theorem}
    Let $a$ be a sink of $Q$, $\flvec{d}$ be a filtration and $M \in \Rep[Q]{\dvec{d}^\nu}\lrang{a}$.
  The map
  \begin{alignat*}{2}
    S^+_a &\colon\quad & \Fl[Q]{\seqv{\dvec{d}}}{M}\lrang{a} &\rightarrow
    \Fl[\sigma_a Q]{\sigma_a \seqv{\dvec{d}}}{S^+_a M}\lrang{a}\\
    && \seqv{U} & \mapsto S^+_a \seqv{U}
  \end{alignat*}
  is an isomorphism of varieties with inverse $\hat D \circ S^+_a \circ \hat D = S^-_a$.
  \label{theom:reflflagiso}
\end{theorem}
\begin{proof}
  First, we show that $S^+_a \seqv{U}$ lies in the correct set. Let $\dvec{e}^{\nu-i} = \dvec{d}^\nu - \dvec{d}^i$.
For each $i$, we have the following commutative diagram with exact columns.
  \[
\begin{CD}
      @.          0        @. 0 @. 0 @. \\
    @.          @VVV                        @VVV           @VVV     @.\\
    0 @>>>  (S^+_a U^{i})_a      @>>> \bigoplus\limits_{j \rightarrow a} U^{i}_j @>{\phi^{U^{i}}_a}>> U^{i}_a @>>> 0\\
    @.          @VVV                        @VVV           @VVV     @.\\
  0 @>>> (S^+_a M)_a @>>> \bigoplus\limits_{j \rightarrow a} M_j @>\phi^{M}_a>> M_a @>>> 0\\
    @.          @VVV                        @VVV           @VVV     @.\\
   0 @>>> (S^+_a M / S^+_a U^{i})_a  @>>> \bigoplus\limits_{j \rightarrow a} (M/U^{i})_j @>\phi^{M/U^{i}}_a>> (M/U^{i})_a @>>> 0\\
    @.          @VVV                        @VVV           @VVV     @.\\
     @.          0        @. 0 @. 0 @.
\end{CD}
  \]
  The two top rows are exact since $\seqv{U} \in \Fl[Q]{\seqv{\dvec{d}}}{M}\lrang{a}$. By the snake lemma,
  we have that the
  bottom row is exact. Therefore, the map
  \[ (S^+_a M / S^+_a U^{i})_a \rightarrow \bigoplus_{j\rightarrow a} (M/U^{i})_j \]
  is injective and hence $S^+_a(M)/S^+_a(U^i) \in \Rep[\sigma_a Q]{\sigma_a(\dvec{e}^{\nu-i})}\lrang{a}$.
  The diagram also yields that $\hat D \circ S^+_a \circ \hat D \circ S^+_a = \id$.
  Since $S^+_a$ is a functor and all choices where natural, we have that $S^+_a$ gives
  a natural transformation between the functors of points of these two varieties. Therefore,
  it is a morphism of varieties.

  Dually, $S^+_a \circ \hat D \circ S^+_a \circ \hat D = \id$. This concludes the proof.
%
\end{proof}
\section{Reflection functors and Hall numbers}
First we recall the definition of the Hall algebra.
The main references for this are Ringel \cite{Ringel_hallalgsandquantumgroups}
and
Hubery \cite{Hubery_ringelhall}.
Let $k$ be a finite field. Let $M, N, X \in \repK{Q}{k}$ be three
$k$-representations of $Q$. Then define
\[ F^X_{M N} := \#\Set{ U \le X | U \text{ subrepresentation}, U \cong N, X/U \cong M }. \]
This is a finite number.

Let $\mathcal{H}_k(Q)$ be the $\Q$-vector space with basis $u_{[X]}$ where $[X]$ is
the isomorphism class of $X$. For convenience we write $u_X$ instead of $u_{[X]}$. Define
\[
u_{[M]} \diamond u_{[N]} := \sum_{[X]} F^{X}_{M N} u_{[X]}.
\]
Then $(\mathcal{H}(\mathcal{A}), +, \diamond)$ is an associative
$\Q$-algebra with unit $1=u_0$, the Ringel-Hall algebra
or just Hall algebra. The composition algebra
is the subalgebra $\mathcal{C}_k(Q)$ of $\mathcal{H}_k(Q)$ generated by
the simple objects without self-extensions. Note that $\mathcal{H}_k(Q)$ and
$\mathcal{C}_k(Q)$ are naturally graded by dimension vector.

To each vertex $i \in Q_0$ there is a simple object $S_i$
given by $(S_i)_i = k$, $(S_i)_j = 0$ for $i \neq j$ and $(S_i)_\alpha = 0$ for all $\alpha \in Q_1$.
We set $u_i := u_{S_i}$ for each $i \in Q_0$. If $w = (i_1, \dots, i_r)$ is
a word in vertices of $Q$, we define
\[ u_w := u_{1} \diamond \dots \diamond u_{r}.\]
By definition, there are numbers $F_w^X$ for each $k$-representation $X$ of $Q$ such that
\[ u_w = \sum_{X} F_w^X u_X. \]
\label{sec:reflhallnum}

Let $k=\F_q$ be the finite field with $q$ elements and
$Q$ a quiver. Let $w=(i_r, \dots, i_1)$ be a word in vertices of $Q$. Define a filtration $\flvec{d}(w)$ by letting
\[\dvec{d}(w)^k := \sum_{j=1}^k \epsilon_{i_k}.\]
Then we obviously have $F_w^X = \# \Fl[Q]{\flvec{d}(w)}{X}$. Therefore, coefficients
in the Hall algebra are closely related to counting points of quiver
flags over finite fields. In the following, we will use reflection functors
to simplify the problem of counting the number of points modulo $q$. As an
application, we will show that for a preprojective or preinjective representation $X$
we have that $\# \Fl[Q]{\flvec{d}}{X} = 1 \mod q$ if $\Fl[Q]{\flvec{d}}{X}$ is non-empty.

\begin{lemma}
    Let $a$ be a sink of $Q$, $k$ a field, $\flvec{d}$ a filtration and $M \in \RepK[Q]{\dvec{d}^\nu}{k}$.
    Then
	\[
    \# \Fl[Q]{\seqv{\dvec{d}}}{M}
	=\sum_{\seqv{r}\ge 0} \# \Gr[A_\nu]{\seqv{r}}{X^{\seqv{r}, \seqv{e}_a}}
    \# \Fl[Q]{\seqv{\dvec{d}} - \seqv{r} \epsilon_a}{\pi_a M}\lrang{a}
    \]
    (on both sides we possibly have $\infty$).

    Moreover, for each sequence of non-negative integers $\seqv{r}$, if
    $
    \Fl[Q]{\seqv{\dvec{d}} - \seqv{r} \epsilon_a}{\pi_a M}\lrang{a}
    $
    is non-empty,
    then so is
    $
    \Gr[A_\nu]{\seqv{r}}{X^{\seqv{r}, \seqv{e}_a}}.
    $
    \label{reflflag:lmm:decompflag}
\end{lemma}
\begin{proof}
    We have that
    \[
    \Fl[Q]{\seqv{\dvec{d}}}{M} = \coprod_{\seqv{r}\ge 0} \Fl[Q]{\seqv{\dvec{d}}}{M}\lrang{a}^{\seqv{r}}.
    \]
    By theorem \ref{theom:qgrass_to_red_fib}, we have for each sequence of non-negative integers $\seqv{r}$ that
    \[
	\# \Fl[Q]{\seqv{\dvec{d}}}{M}\lrang{a}^{\seqv{r}}
	=
	\# \Gr[A_\nu]{\seqv{r}}{X^{\seqv{r}, \seqv{e}_a}} \# \Fl[Q]{\seqv{\dvec{d}} - \seqv{r} \epsilon_a}{\pi_a M}\lrang{a}.
    \]
    By the same theorem we have that if 
    $\Fl[Q]{\seqv{\dvec{d}} - \seqv{r} \epsilon_a}{\pi_a M}\lrang{a}$ is non-empty, then so
    is $\Gr[A_\nu]{\seqv{r}}{X^{\seqv{r}, \seqv{e}_a}}$.
\end{proof}
\begin{lemma}
    Let $a$ be a sink of $Q$, $k$ a field, $\flvec{d}$ a filtration
    and $M \in \RepK[Q]{\dvec{d}^\nu}{k}\lrang{a}^s$. Let $\seqv{r}_+= \seqv{r}_+(\flvec{d})$ be given as follows:
    \begin{align*}
        r^0_+ &:= 0; &\\
        r^i_+ &:= \max\{0, (\sigma_a(\dvec{d}^{i-1} - \dvec{d}^i))_a + r^{i-1}_+\}& \text{for } 0 < i < \nu;\\
        r^\nu_+ &:= s.
    \end{align*}
    Now let $\seqv{r}$ be a sequence of integers. If
    $\Fl[Q]{\seqv{\dvec{d}}}{M}\lrang{a}^{\seqv{r}}$ is non-empty, then $\seqv{r} \ge \seqv{r}_+$.
    \label{reflflag:lmm:rplusmin}
\end{lemma}
\begin{proof}
    Let $\seqv{U} \in \Fl[Q]{\seqv{\dvec{d}}}{M}\lrang{a}^{\seqv{r}}$. By definition,
    $\Hom(U^i, S_a) = r^i$. We have
    \[r^i= \codim \Bild \Phi_a^{U^i} = \dim \ker \Phi_a^{U^i} + d^i_a - \sum_{j \rightarrow a} d^i_j
    = \dim \ker \Phi_a^{U^i} - (\sigma_a \dvec{d}^i)_a.\]
    We prove $r^i \ge r_+^i$ by induction on $i$. For $i=0$ the claim
    is obviously true. Now let $0\le i \le \nu-2$. Obviously, $\dim \ker \Phi_a^{U^i} \le \dim \ker \Phi_a^{U^{i+1}}$.
    Therefore,
    \[ r^i_+  + (\sigma_a \dvec{d}^i)_a 
    \le r^i  + (\sigma_a \dvec{d}^i)_a= \dim \ker \Phi_a^{U^i} \le \dim \ker \Phi_a^{U^{i+1}} =
    r^{i+1} + (\sigma_a \dvec{d}^{i+1})_a.\]
    Hence, $r^{i+1} \ge \max\{0,(\sigma_a(\dvec{d}^{i} - \dvec{d}^{i+1}))_a + r^i_+\} = r^{i+1}_+$.

    For $r^\nu_+$ note that, by definition, $U^\nu = M$ and therefore
    \[ r^\nu =\codim \Bild \Phi_a^{U^\nu} = \codim \Bild \Phi_a^{M} = s.\]
\end{proof}
\begin{remark}
    Note that $\seqv{r}_+(\flvec{d}-\seqv{r}_+(\flvec{d}) \epsilon_a)=0$ since
    \[
    \sigma_a (\dvec{d}^i - \dvec{d}^{i-1})_a + r^i_+(\flvec{d}) - r^{i-1}_+(\flvec{d})
    \ge r^i_+(\flvec{d}) - r^i_+(\flvec{d}) = 0.
    \]
    For any filtration $\flvec{d}$ of some representation $M$ it is enough to remember the terms
    \[
    (\dvec{d}^1, \dots, \dvec{d}^{\nu-1})
    \]
    since we always have $\dvec{d}^0 = 0$
    and $\dvec{d}^\nu = \dimve M$. Note that the rule to construct $r_+^i$ for
    $0<i<\nu$
    depends neither on $\dvec{d}^0$ nor on $\dvec{d}^\nu$. Therefore, we
    can define
    \[S_a^+ \flvec{d} = S_a^+ (\dvec{d}^1, \dots, \dvec{d}^{\nu-1}) :=
    (\sigma_a \dvec{d}^1 + r_+^1 \epsilon_a, \dots, \sigma_a \dvec{d}^{\nu-1} + r_+^{\nu-1}). \]
    If $\flvec{d}$ is
    a filtration of $M$, then $S_a^+ \flvec{d}$ is a filtration of $S_a^+ M$ if and only if
    $(S_a^+ \flvec{d})^{\nu-1} \le \dimve S_a^+ M$.
    \label{reflflag:rem:reflfilt}
\end{remark}
\begin{corollary}
    Let $a$ be a sink, $k$ a field, $\flvec{d}$ a filtration and $M \in \RepK[Q]{\dvec{d}^\nu}{k}\lrang{a}^s$.
    Then
	\[
    \# \Fl[Q]{\seqv{\dvec{d}}}{M}
    =\sum_{\seqv{r}\ge 0} \# \Gr[A_\nu]{\seqv{r}+\seqv{r}_+}{X^{\seqv{r}+\seqv{r}_+, \seqv{e}_a}}
    \# \Fl[\sigma_a Q]{S_a^+\seqv{\dvec{d}}+ \seqv{r}\epsilon_a}{S_a^+ M}\lrang{a}.
    \]

    In particular, if $k$ is a finite field of cardinality $q$, we have
    \[
    \# \Fl[Q]{\seqv{\dvec{d}}}{M}
    =\sum_{\seqv{r}\ge 0}
    \# \Fl[\sigma_a Q]{S_a^+\seqv{\dvec{d}}+\seqv{r}\epsilon_a}{S_a^+ M}\lrang{a} \mod q.
    \]
    \label{reflflag:cory:decompreflmodq}
\end{corollary}
\begin{proof}
    By lemmas \ref{reflflag:lmm:decompflag} and \ref{reflflag:lmm:rplusmin} we obtain
    that
	\[
    \# \Fl[Q]{\seqv{\dvec{d}}}{M}
    =\sum_{\seqv{r}\ge 0} \# \Gr[A_\nu]{\seqv{r}+\seqv{r}_+}{X^{\seqv{r}+\seqv{r}_+, \seqv{e}_a}}
    \# \Fl[Q]{\flvec{d} - (\seqv{r} + \seqv{r}_+)\epsilon_a}{\pi^s_a M}\lrang{a}.
    \]
    Note that $\sigma_a (\dvec{d}^i - (r^i+r^i_+)\epsilon_a) =
    (S_a^+\flvec{d})^i + r^i \epsilon_a$ for all $0 < i < \nu$. Therefore, theorem \ref{theom:reflflagiso}
    yields that
    \[
    \Fl[Q]{\flvec{d} - (\seqv{r} + \seqv{r}_+)\epsilon_a}{\pi^s_a M}\lrang{a} \cong
    \Fl[\sigma_a Q]{S_a^+\flvec{d} + \seqv{r}\epsilon_a}{S_a^+ M}\lrang{a}.
    \]
    This proves the first claim.

    Now let $k$ be a finite field of cardinality $q$. If
    $\Gr[A_\nu]{\seqv{r}+\seqv{r}_+}{X^{\seqv{r}+\seqv{r}_+, \seqv{e}_a}}$
    is non-empty, then its number is one modulo $q$ by lemma \ref{lmm:flagfib_not_empty}.
    The second part of lemma \ref{reflflag:lmm:decompflag} yields that whenever
    $\Fl[Q]{\flvec{d} - (\seqv{r} + \seqv{r}_+)\epsilon_a}{\pi^s_a M}\lrang{a}$
    is non-empty, then $\Gr[A_\nu]{\seqv{r}+\seqv{r}_+}{X^{\seqv{r}+\seqv{r}_+, \seqv{e}_a}}$
    is non-empty. This finishes the proof.
\end{proof}
We obtain the following.
\begin{theorem}
    \label{theom:reflflagmodq}
    Let $a$ be a sink, $k$ a field, $\flvec{d}$ a filtration and $M \in \RepK[Q]{\dvec{d}^\nu}{k}\lrang{a}^s$.
	Then $\Fl[Q]{\seqv{\dvec{d}}}{M}$ is empty if and only if
    $\Fl[\sigma_a Q]{S^+_a \flvec{d}}{S^+_a M}$ is.

	Moreover, if $k=\F_q$ is a finite field, then
    \[
	\# \Fl[Q]{\seqv{\dvec{d}}}{M}=
	\# \Fl[\sigma_a Q]{S^+_a \flvec{d}}{S^+_a M} \mod q.
    \]
\end{theorem}
\begin{proof}
    By corollary \ref{reflflag:cory:decompreflmodq} we obtain that
	\[
    \# \Fl[Q]{\seqv{\dvec{d}}}{M}
    =\sum_{\seqv{r}\ge 0} \# \Gr[A_\nu]{\seqv{r}+\seqv{r}_+}{X^{\seqv{r}+\seqv{r}_+, \seqv{e}_a}}
    \# \Fl[\sigma_a Q]{S_a^+\seqv{\dvec{d}}+ \seqv{r}\epsilon_a}{S_a^+ M}\lrang{a}.
    \]
    Note that $S_a^+ \flvec{d}$ is a filtration of $S_a^+M$ if and only if
    $(S_a^+ \flvec{d})^{\nu-1} \le \dimve S_a^+ M$. Therefore, if $S_a^+ \flvec{d}$
    is not a filtration of $S_a^+ M$, then 
    each $\Fl[\sigma_a Q]{S_a^+\seqv{\dvec{d}}+ \seqv{r}\epsilon_a}{S_a^+ M}\lrang{a}$
    is empty for all $\seqv{r} \ge 0$ and hence, so is 
    $\Fl[Q]{\seqv{\dvec{d}}}{M}$. In this case, we also have that
    $\Fl[\sigma_a Q]{S^+_a \flvec{d}}{S^+_a M}$ is empty. Both claims follow.

    Assume now that $S_a^+ \flvec{d}$ is a filtration of $S_a^+M$.
    We have that
    \[\Fl[\sigma_a Q]{S^+_a \flvec{d}}{S^+_a M} \cong \Fl[\sigma_a Q^{op}]{\dimve S^+_a M - \rever{S^+_a \flvec{d}}}{DS^+_a M}\]
    via $\hat D$. Let $\flvec{f}:=\dimve S^+_a M - \rever{S^+_a \flvec{d}}$. By using lemma
    \ref{reflflag:lmm:decompflag} we obtain that
    \[
    \# \Fl[\sigma_a Q^{op}]{\flvec{f}}{DS^+_a M} 
    =\sum_{\seqv{r}\ge 0} \# \Gr[A_\nu]{\rever{\seqv{r}}}{X^{\rever{\seqv{r}}, (S^+_a\flvec{d})_a}}
    \# \Fl[\sigma_a Q^{op}]{\flvec{f} - \rever{\seqv{r}} \epsilon_a}{D S_a^+ M}\lrang{a}.
    \]
    Moreover, by the same lemma we have for each $\seqv{r}\ge 0$ that if
    $\Fl[\sigma_a Q^{op}]{\flvec{f} - \rever{\seqv{r}} \epsilon_a}{D S_a^+ M}\lrang{a}$ is
    non-empty, then $\Gr[A_\nu]{\rever{\seqv{r}}}{X^{\rever{\seqv{r}}, (S^+_a\flvec{d})_a}}$ is non-empty.
    Using $\hat D$ yields
    \[
    \Fl[\sigma_a Q^{op}]{\flvec{f} - \rever{\seqv{r}} \epsilon_a}{D S_a^+ M}\lrang{a}
    \cong \Fl[\sigma_a Q]{S_a^+\flvec{d} + \seqv{r}\epsilon_a}{S_a^+ M}\lrang{a}.
    \]
    Combining these equalities, we have that
    \[
    \# \Fl[\sigma_a Q]{S^+_a \flvec{d}}{S^+_a M}
    =\sum_{\seqv{r}\ge 0} \# \Gr[A_\nu]{\rever{\seqv{r}}}{X^{\rever{\seqv{r}}, (S^+_a\flvec{d})_a}}
    \# \Fl[\sigma_a Q]{S_a^+\flvec{d} + \seqv{r}\epsilon_a}{S_a^+ M}\lrang{a}.
    \]
    Therefore, $\Fl[\sigma_a Q]{S^+_a \flvec{d}}{S^+_a M}$ is empty if and only if for all
    $\seqv{r} \ge 0$ we have that the variety $\Fl[\sigma_a Q]{S_a^+\flvec{d} + \seqv{r}\epsilon_a}{S_a^+ M}\lrang{a}$
    is empty. The same is true for $\Fl[Q]{\seqv{\flvec{d}}}{M}$ and this proves the first claim.

    Now let $k$ be  a finite field with $q$ elements. By the first part, if 
    $\Fl[\sigma_a Q]{S_a^+\flvec{d} + \seqv{r} \epsilon_a}{S_a^+ M}\lrang{a}$
    is non-empty, then
    $\Gr[A_\nu]{\rever{\seqv{r}}}{X^{\rever{\seqv{r}}, (S^+_a\flvec{d})_a}}$ is non-empty.
    Therefore, lemma \ref{lmm:flagfib_not_empty} yields that
    \[\# \Fl[\sigma_a Q]{S^+_a \flvec{d}}{S^+_a M}
    =\sum_{\seqv{r}\ge 0}
    \# \Fl[\sigma_a Q]{S_a^+\flvec{d} + \seqv{r}\epsilon_a}{S_a^+ M}\lrang{a} \mod q.\]
    By corollary \ref{reflflag:cory:decompreflmodq} this is equal to
    $\# \Fl[Q]{\flvec{d}}{M}$. This finishes the proof.
\end{proof}

\begin{remark}
    The Coxeter functor $C^+$ is by definition the composition of reflection functors
    associated to an admissible ordering
    $(a_1, \dots, a_n)$ of $Q$. The action on a filtration, which we also denote by $C^+$, is
    given by $C^+ \flvec{d} := S^+_{a_n} \dots S^+_{a_1} \flvec{d}$. It is not clear
    that $C^+$ on a filtration does not depend on the choice of the admissible ordering.
\end{remark}
We immediately obtain the following.
\begin{corollary}
    \label{reflflags:cory:preprojone}
    Let $M$ be a preprojective $k$-representation and
    let $\flvec{d}$ be a filtration of $\dimve M$. Take $r\ge0$ such that
    $(C^+)^r M = 0$.
    
    Then
    $\Fl[Q]{\seqv{\dvec{d}}}{M}$ is non-empty if and only if
    we have that $(C^+)^r \flvec{d} =0$ and that for every intermediate sequence $w$ of admissible sink reflections
    $S^+_w \flvec{d}$ is a filtration of $S^+_w M$. In particular, this depends only
    on the isomorphism class of $M$ and the filtration $\flvec{d}$, but not on the choice
    of $M$ or the field $k$.

    Moreover, if $k$ is a finite field with $q$ elements, then
    $\Fl[Q]{\seqv{\dvec{d}}}{M}$ non-empty implies that
    \[\# \Fl[Q]{\seqv{\dvec{d}}}{M} = 1 \mod q.\]
\end{corollary}
\begin{proof}
    Using remark \ref{reflflag:rem:reflfilt} we obtain that for each reflection at a sink
    $a$ of $Q$ we have that
    $S^+_a \flvec{d}$ is again a filtration of $S^+_a M$ if and only if
    $(S^+_a \flvec{d})^{\nu-1} \le \dimve S^+_a M$.
    If this is not the case, then
    the quiver flag is empty by theorem \ref{theom:reflflagmodq}.
    Therefore, if the quiver flag is non-empty, then for every
    intermediate sequence $w$ of admissible sink reflections
    we have that $S^+_w \flvec{d}$ is a filtration of $S^+_w M$. We call this condition (*).

    Assume that (*) holds. Iteratively applying
    theorem \ref{theom:reflflagmodq} we have that
    $\Fl[Q]{\flvec{d}}{M}$ is empty if and only if $\Fl[Q]{(C^+)^r\flvec{d}}{(C^+)^r M} = \Fl[Q]{(C^+)^r\flvec{d}}{0}$
    is empty.
    There is only
    one filtration of the $0$ representation, namely $(0,0,\dots,0)$. Therefore,
    $\Fl[Q]{\flvec{d}}{M}$ is non-empty if and only if $(C^+)^r\flvec{d} = 0$.
    This proves the first part since we already have seen that if (*) does not hold, then
    $\Fl[Q]{\flvec{d}}{M}$ is empty.
    
    Assume now that $k$ is a finite field with $q$ elements. If (*) does not
    hold, then the quiver flag is empty and the claim holds.
    Assume therefore that (*) holds. As before, applying
    theorem \ref{theom:reflflagmodq} yields that
    \[ \#\Fl[Q]{\flvec{d}}{M} = \#\Fl[Q]{(C^+)^r\flvec{d}}{(C^+)^r M} = \#\Fl[Q]{(C^+)^r\flvec{d}}{0} \mod q.\]
    There is only
    one filtration of the zero representation, namely $(0,0,\dots,0)$, and
    the number of flags of this type is obviously equal to one.
    This concludes the proof.
\end{proof}
\section{Dynkin case}
In this section let $Q$ be
a Dynkin quiver. We first introduce the generic Hall algebra $\mathcal{H}_q(Q)$ and the
composition monoid $\mathcal{CM}(Q)$ and then use the results of the previous section to show that
$\mathcal{H}_0(Q) \cong \mathcal{CM}(Q)$, where the isomorphism is given by sending $u_i$ to
$S_i$.

Since $Q$ is Dynkin, every representation is preprojective. Moreover,
indecomposable representations of $Q$ are in bijection with
the set of positive roots $\Delta_+$ of the Lie algebra corresponding to the underlying Dynkin diagram.
An isomorphism class is therefore given by a function from $\Delta_+$ to $\N$ with finite support. We denote this
set by $\Phi$. For each element $\mu \in \Phi$ and each field $k$ we can choose a $k$-representation
$M(\mu, k)$ having isomorphism class $\mu$.

Hall polynomials exist
with respect to $\Phi$, as shown by Ringel \cite{Ringel_hallpolysforrepfiniteheralgs}.
More precisely, for $\mu, \nu, \xi \in \Phi$ there is a polynomial
$f^{\xi}_{\mu \nu} (q) \in \Z[q]$ such that for each finite field $k$ with $q_k$ elements
we have
\[ F^{M(\xi,k)}_{M(\mu,k) M(\nu,k)} = f^{\xi}_{\mu \nu} (q_k).\]

We define the generic Hall algebra $\mathcal{H}_q (Q)$ to be the free $\Z[q]$-module with basis
$\Set {u_{\alpha} | \alpha \in \Phi}$ and multiplication given by:
\[u_{\mu} \diamond u_{\nu} = \sum_{\xi} f^{\xi}_{\mu \nu} (q) u_{\xi}. \]
The generic composition algebra $\mathcal{C}_q(Q)$ is the subalgebra of $\mathcal{H}_q(Q)$ generated
by the simple representations without self-extensions, or more precisely their isomorphism
classes. For an acyclic quiver theses are exactly the $u_i$. If the quiver is fixed, then we often write $\mathcal{H}_q$ and
$\mathcal{C}_q$ instead of $\mathcal{H}_q(Q)$
and $\mathcal{C}_q(Q)$. Moreover, for convenience we identify for any representation $M \cong M(\alpha, k)$, $u_M$
with $u_\alpha$ and $f^{X}_{M N}$ with $f^{\xi}_{\mu \nu} (q)$.

When calculating Hall polynomials, certain quantum numbers appear. Let $R$ be some commutative ring
and let $q \in R$. Usually $R$ will be $\Z[q]$, the polynomial ring in one variable. We define for $r,n \in \N$,
$0 \le r \le n$:
\begin{align*}
    [n]_q &:=  1 + q + \dots + q^{n-1} \\
    [n]_q ! &:= \prod_{i=1}^{n} [i]_q \\
    \qbinom{n}{r}_q &:= \frac{[n]_q !}{[r]_q ! [n-r]_q!}.
\end{align*}
Obviously, $\qbinom{n}{r}_0 = 1$.

Fix
an algebraically closed field $k$.
We say that a representation $M$ degenerates to $N$, $M \ledeg N$, if
$\Orbit_N \subseteq \overline{\Orbit_M}$, where we take the closure in the Zariski topology.

For two arbitrary sets $U \subseteq \Rep{\dvec{d}}, V \subseteq \Rep{\dvec{e}}$ we define
\begin{align*}
	\mathcal{E}(U,V) := \{ M \in \Rep{\dvec{d} + \dvec{e}} \: |& \: \exists \:A \in U, B \in V \
	\text{and a short exact sequence } \\
	 & 0 \rightarrow B \rightarrow M \rightarrow A \rightarrow 0 \}.
\end{align*}

The multiplication on closed irreducible $\GL_\dvec{d}$-stable respectively $\GL_\dvec{e}$-stable
subvarieties $\mathcal{A} \subseteq \Rep{\dvec{d}},
\mathcal{B} \subseteq \Rep{\dvec{e}}$ is defined as:
\[ \mathcal{A} * \mathcal{B} := \mathcal{E}(\mathcal{A}, \mathcal{B}).\]
The set of closed irreducible subvarieties of nilpotent representations with
this multiplication is a monoid with unit $\Rep{\dvec{0}}$, the generic extension monoid $\mathcal{M}(Q)$. The composition monoid
$\mathcal{CM}(Q)$ is the submonoid
generated by the orbits of simple representations without self-extensions. All this is due to Reineke \cite{Reineke_monoid}.
For any word $w = (i_1, \dots, i_r)$ in vertices of $Q$ we define
$\mathcal{A}_w := \Orbit_{S_1} * \dots * \Orbit_{S_r}$. This is an element of $\mathcal{CM}(Q)$.

We can now use the machinery we developed in the previous section to prove
that, for $Q$ a Dynkin quiver, the generic composition algebra
specialised at $q=0$
and the composition monoid are isomorphic.

\begin{proposition}
    \label{reflflags:propn:dynkinallone}
    Let $X$ be a $k$-representation of $Q$ and $w$ a word in vertices of $Q$.
    Then the condition that $X$
    has a filtration of type $w$ only depends on $w$ and $[X] \in \Phi$ and not on the choice of $X$ or the field
    $k$.

    Moreover, we have that
    \[ u_w = \sum_{[X] \in [\mathcal{A}_w]} u_{[X]} \in \mathcal{H}_0(Q). \]
\end{proposition}
\begin{proof}
    Since all representations of $Q$ are preprojective, the first part of the
    statement follows directly from corollary \ref{reflflags:cory:preprojone}.
    Therefore, the sum in the second part is well-defined (i.e. the set
    $[\mathcal{A}_w]$ does not depend on the field).

    If $k$ is a finite field with $q$ elements, corollary \ref{reflflags:cory:preprojone} also yields that
    \[ F_w^X = \# \Fl[Q]{\flvec{d}(w)}{X}  = \begin{cases}
        1 \mod q & \text{if } X \in \mathcal{A}_w,\\
        0  & \text{else.}
    \end{cases}
    \]
    Since $F_w^X= f_w^{[X]}(q)$ and we just showed that this is one modulo $q$ for all prime powers $q$
    we have that
    \[f_w^{[X]}(0) =\begin{cases}
        1 & \text{if } [X] \in [\mathcal{A}_w],\\
        0 & \text{else.}
    \end{cases}
    \]
    This yields the claim.
\end{proof}
We obtain the main theorem for the Dynkin case.
\begin{theorem}
  The map
  \begin{alignat*}{2}
    \Psi &\colon\quad & \Q\mathcal{M}(Q) & \rightarrow \mathcal{H}_0(Q)\\
    && \mathcal{A} & \mapsto \sum\limits_{[M] \in [\mathcal{A}]} u_{[M]}
  \end{alignat*}
  is an isomorphism of $\Q$-algebras.
  \label{reflflags:theom:dynkin}
\end{theorem}
\begin{proof}
    Note that for $Q$ Dynkin we have that $\mathcal{M}(Q) \cong \mathcal{CM}(Q)$
    and $\mathcal{H}_q(Q) \cong \mathcal{C}_q(Q)$. Therefore,
    for each $\mathcal{A} \in \mathcal{M}(Q)$ there is a word $w$ in
    vertices of $Q$ such that $\mathcal{A} = \mathcal{A}_w$. In
    the previous proposition we showed that the map sending $\mathcal{A}_w$ to
    \[\Psi(\mathcal{A}_w) = \sum_{[M] \in [\mathcal{A}_w]} u_{[M]} = u_w\]
    is well-defined. Therefore, $\Psi$ is a homomorphism, since
    \[\Psi(\mathcal{A}_{w} * \mathcal{A}_{v}) = \Psi(\mathcal{A}_{wv}) =
    u_{wv}=u_{w} \diamond u_v = \Psi(\mathcal{A}_{w}) \diamond \Psi(\mathcal{A}_{v}).  \]
    
%
    $\Psi$ is surjective since it is a homomorphism, and the generators
    $u_i$ of $\mathcal{H}_0(Q)$ are in the image of $\Psi$. More precisely,
    $\Psi(\Orbit_{S_i}) = u_i$. Obviously, $\Psi$ is a graded morphism
    of graded algebras. The dimension of the $\dvec{d}$-th graded part
    of $\Q\mathcal{M}(Q)$ is the same as the dimension
    of the $\dvec{d}$-th graded part of $\mathcal{H}_0(Q)$, namely the number of isomorphism classes of representations
    of dimension vector $\dvec{d}$. Since each graded part is finite dimensional and
    $\Psi$ is surjective, we have
    that $\Psi$ is an isomorphism.
\end{proof}
\section{Further work}
In an upcoming paper, we show that if $Q$ is an acyclic extendend Dynkin quiver, then
\begin{alignat*}{2}
\Phi&\colon\quad& \mathcal{C}_0(Q) &\rightarrow \mathcal{CM}(Q)\\
  && u_i &\rightarrow S_i
\end{alignat*}
is a homomorphism and give generators for its
non-trivial kernel. We do this by a close analysis of the geometric version
of reflection functors and the result of a previous paper on the Hall algebra of an oriented cycle.
\bibliography{repsofalgs}
\end{document}